\documentclass[english]{smfart}
\usepackage{sabbah_deformations-irreg-singularities}

\begin{document}
\frontmatter
\title[Deformations and degenerations of irregular singularities]{Integrable deformations and degenerations of~some~irregular~singularities}

\author[C.~Sabbah]{Claude Sabbah}
\address[C.~Sabbah]{CMLS, CNRS, École polytechnique, Institut Polytechnique de Paris\\
F--91128 Palaiseau cedex\\
France}
\email{Claude.Sabbah@polytechnique.edu}
\urladdr{http://www.math.polytechnique.fr/perso/sabbah}


\begin{abstract}
We extend to a degenerate case a result of B.\,Malgrange on integrable deformations of irregular singularities, inspired by an article of G.\,Cotti, B.\,Dubrovin and D.\,Guzzetti \cite{C-D-G17a}. We give an application to integrable deformations of some meromorphic connections in Birkhoff normal form and to the construction of Frobenius manifolds.
\end{abstract}

\maketitle
\vspace*{-2\baselineskip}
\tableofcontents
\mainmatter

\section{Introduction}

\subsection*{Motivations}
Let us consider a family of differential systems of a complex variable~$z$ parametrized by $t$ with matrix $A(t,z)\,\rd z/z$ given by
\begin{equation}\label{eq:system}
A(t,z)=\frac{A_0(t)}z+A_1(t).
\end{equation}
We assume that $A_0(t),A_1(t)$ are square matrices depending holomorphically on the parameter $t\in T$. Assume for example that $A_0(t)$ is non-resonant for generic values of $t$, that is, its eigenvalues are pairwise distinct. Then one can locally take these eigenvalues as parameters, and a suitable base change, formal with respect to $z$ and locally holomorphic with respect to $t$, reduces the system into a diagonal form, so that its $z$-formal solutions are easy to obtain.

If on the other hand some eigenvalues of $A_0(t)$ coincide at some $t_o\in T$, the $z$-formal behaviour of solutions in the neighbourhood of $t_o$ becomes much harder to understand. The holomorphic behaviour of solutions is the subject of singular perturbation theory, and is addressed in an extensive literature, and the question is: how much do the Stokes matrices of the one-variable system at $t=t_o$ determine the Stokes matrices at neighbouring points?

Adding an integrability assumption to \eqref{eq:system} makes the complexity of the problem bounded, in the following sense. If we assume that there exists an integrable meromorphic connection (on the trivial vector bundle) with respect to all variables $z,t$ so that the original system \eqref{eq:system} is the $z$-part of this connection, then it is worthwhile treating all variables on the same footing, allowing meromorphic changes of all the variables, also called complex blowing-ups. The system \eqref{eq:system} can then be reduced to a simpler one after a suitable meromorphic change of variables: this is the content of a fundamental theorem of K.\,Kedlaya \cite{Kedlaya09,Kedlaya10} in the present setting, while one can also refer to the work of T.\,Mochizuki \cite{Mochizuki07b,Mochizuki08} in an algebraic setting (with respect to $z$ and $t$). The price to pay is the introduction of singularities in the set of poles of the system, and the loss of the distinction between the notion of a variable and that of a parameter.\enlargethispage{\baselineskip}%

The integrability property can arise in at least two ways on a given system \eqref{eq:system}.
\begin{itemize}
\item
An integrability property at a generic point of $t$ may propagate all along the set of parameters where the system is defined.
\item
The system \eqref{eq:system} is obtained as the solution to an isomonodromy deformation problem.
\end{itemize}

In the very interesting and inspiring article \cite{C-D-G17a} (\cf also \cite{C-G17b,C-G17a}), G.\,Cotti, B.\,Dubrovin and D.\,Guzzetti have analyzed with much care the case where $T$ parametrizes the eigenvalues of the matrix $A_0$, which is then assumed to be the matrix $\diag(t_1,\dots,t_n)$, and $t_o$ belongs to the union $\Delta$ of diagonal hyperplanes (coalescing eigenvalues). In particular, they have shown how Stokes data at $t_o$ can be extended in some neighbourhood of $t_o$, and they gave applications to Frobenius manifolds.

Our aim is to revisit some of their results from a different point of view, namely that of isomonodromic deformations of irregular differential equations, with the geome\-tric perspective of \cite{Malgrange83dbb}. We will take advantage of the results of K.\,Kedlaya and T.\,Mochizuki mentioned above on meromorphic connections in dimension $\geq2$. In this setting, the behaviour of Stokes matrices is sufficiently well understood after a suitable meromorphic base change of variables, under the name of sheaf of Stokes torsors, which we will explain in Section \ref{subsec:StT}.

\subsection*{The setting}
Let $T=\CC^n$ with coordinates $t_1,\dots,t_n$ and let $X$ be a neighbourhood of $T\times\{0\}$ in $T\times\CC$. We equip the extra factor $\CC$ with the coordinate $z$, and we regard $T\subset X$ as the smooth hypersurface defined by the equation $z=0$.

We consider as a \emph{model system} the system \eqref{eq:system} where the matrix $A_0(t)$ is the block-diagonal matrix $\Lambda(t)$ with diagonal blocks $t_i\id$, and the matrix $A_1(t)$ is constant and block diagonal correspondingly. This system is integrable, and the $T$-component of the connection has a block-diagonal matrix with blocks $(-1/z)\rd t_i\cdot \id$.

We aim at analyzing integrable systems \eqref{eq:system} that are formally isomorphic, as integrable systems, to the model system above. We will keep the formal isomorphism as part of the data, in other words, we are mainly interested in all possible Stokes data that can occur on such a model system.

We will use the language of meromorphic flat bundles, which happens to be more flexible when considering the meromorphic base changes like complex blowing-ups. 

For each $i\in\{1,\dots,n\}$, let us be given a locally free $\cO_X(*T)$-module $\cR_i$ of finite rank, endowed with an integrable meromorphic connection $\nabla$ (an object we call a \emph{$T$\nobreakdash-meromorphic flat bundle}). We assume that $\nabla$ has \emph{regular singularities} along~$T$. We~then set $\cE^{-t_i/z}\otimes\cR_i:=(\cR_i,\nabla-\rd(t_i/z))$, and we consider the model $T$-meromorphic flat bundle
\begin{equation}\label{eq:cN}
\cN:=\bigoplus_{i=1}^n(\cE^{-t_i/z}\otimes\cR_i).
\end{equation}

Let $\cO_{\wh T}$ be the formal completion of $\cO_X$ along $T$ and let us set \hbox{$\cM_{\wh T}:=\cO_{\wh T}\otimes_{\cO_X}\cM$}. We will be concerned with $T$-meromorphic flat bundles $\cM$ endowed with an isomorphism $\iso_{\wh T}:\cM_{\wh T}\isom\cN_{\wh T}$. We note that a morphism $(\cM_1,\iso_{\wh T})\to(\cM_2,\iso_{\wh T})$ between such objects (with the obvious definition) is uniquely determined by the associated formal morphism $\cM_{1\wh T}\to\cM_{2\wh T}$, and there is at most one isomorphism between two objects.

For $t_o\in T$, we denote by $\gamma_{t_o}:\{t_o\}\times(\CC_z,0)\hto X$ the inclusion. Then $\gamma_{t_o}^*\cN$ is a free $\cO_{\CC_z,0}(*0)$-module that can be endowed with the pullback connection, which we denote by $\gamma_{t_o}^+\cN$, and it has a form similar to \eqref{eq:cN} by replacing~$t_i$ with the constant~$t_{o,i}$ and $\cR_i$ with its restriction $\gamma_{t_o}^+\cR_i$. However some of the $t_{o,i}$ may coincide, which leads to grouping the corresponding $\gamma_{t_o}^+\cR_i$. The~most degenerate case is when $t_{o,i}=0$ for all $i$ or, similarly, when all $t_{o,i}$ coincide.

The~nondegenerate case is when $t_o$ does not belong to the union $\Delta$ of the diagonal hyper\-planes in $T$, \ie when the $t_{o,i}$ are pairwise distinct, which is the case considered in \cite{J-M-U81} as well as in \cite{Malgrange83dbb}. In such a case, for any simply connected open subset $U$ of $T\moins\Delta$ and for any $t_o\in U$, the restriction $\gamma_{t_o}^+$ induces a bijection between the set of isomorphism classes of pairs $(\cM_U,\iso_{\wh U})$ and pairs $(\cM^{t_o},\iso_{\wh0})$ on $\CC_z$. This is more classically expressed in terms of Stokes matrices. From the point of view developed here, we interpret this result by saying that, on a simply connected set $U$, giving a global section of a local system (the sheaf of Stokes torsors) is equivalent to giving a germ of section at one point of $U$. Our aim is to show that, for~$\cN$ as in \eqref{eq:cN}, a similar result holds for any $t_o\in \Delta$.

The set $\Delta$ is naturally stratified: the stratum of a point is defined by specifying the precise sets of coordinates that coincide at this point. Given $t_o\in\Delta$, we denote by $S(t_o)$ its stratum. Its closure is a linear subspace of $T$. 

\subsection*{The results}

\begin{theoreme}\label{th:extension}
Let $S_o$ be a stratum of $\Delta$ and let $U$ be an open subset of $T$ such that
\begin{enumeratea}
\item\label{th:extensiona}
$U\cap S_o$ is simply connected,
\item\label{th:extensionb}
$U$ is star shaped with respect to $U\cap S_o$ (\cf Definition \ref{def:starshaped}),
\item\label{th:extensionc}
for every stratum $S$ such that $\ov S\cap S_o=\emptyset$, we also have $U\cap\ov S=\emptyset$.
\end{enumeratea}
Then for any $t_o\in S_o$, the restriction~$\gamma_{t_o}^+$ induces a bijection between the set of isomorphism classes of pairs of $(\cM,\iso_{\wh T})$ defined on $U$ and that of pairs $(\gamma_{t_o}^+\cM,\iso_{\wh0})$.
\end{theoreme}

We can interpret Theorem \ref{th:extension} in two ways.

\begin{corollaire}\label{cor:extension}\mbox{}
\begin{enumerate}
\item
Any pair $(\cM^{t_o},\iso_{\wh 0})$ with formal model $\gamma_{t_o}^+\cN$ on $(\CC_z,0)$ can be extended as a pair $(\cM,\iso_{\wh T})$ on any open set $U$ satisfying \ref{th:extension}\eqref{th:extensiona}--\eqref{th:extensionc}.
\item
Any $(\cM,\iso_{\wh T})_V$ defined on a connected open set $V$ of $T$ is uniquely determined by any of its restrictions $\gamma_{t_o}^+(\cM,\iso_{\wh T})$ at $t_o\in V$.
\end{enumerate}
\end{corollaire}

We will see (Corollary \ref{cor:uniquenessglobalI}) that the second part of the corollary holds in a much more general setting, according to recent results of J.-B.\,Teyssier \cite{Teyssier17}.

As a special case, let us assume that $t_o$ belongs to the smallest stratum, where all coordinates coincide. Then $\gamma_{t_o}^+\cN$ has a regular singularity at $0$ up to an exponential twist, and any $(\cM^{t_o},\iso_{\wh 0})$ is isomorphic to $(\gamma_{t_o}^+\cN,\id)$. We deduce that any $(\cM,\iso_{\wh T})$ defined everywhere on $T$ is isomorphic to $(\cN,\id)$, since $S(t_o)\simeq\CC$ is simply connected and we can apply the theorem with $U=T$.

The theorem has a generalization as follows. Let $Y$ be a $1$\nobreakdash-con\-nec\-ted complex manifold and let $f=(f_1,\dots,f_n):Y\to T$ be a holomorphic mapping. We still denote by $f$ the map $f\times\id:Z:=Y\times\CC\to T\times\CC=X$ when the context is clear. There is no loss of generality in assuming that $f^{-1}(\Delta)$ is a hypersurface in $Y$ (otherwise, one can replace $n$ with a smaller $n'$). Then a $Y$-meromorphic flat bundle on $Z$ with regular singularities is nothing but the pullback of a $T$-meromorphic flat bundle on $X$ with regular singularities, and a $Y$-meromorphic flat bundle of the form
\[
\cN_Y=\bigoplus_{i=1}^n(\cE^{-f_i(y)/z}\otimes\cR_{Y,i})
\]
is isomorphic to $f^+\cN$, for $\cN$ as in \eqref{eq:cN} with $\cR_i$ such that $\cR_{Y,i}=f^+\cR_i$ (where $f^+$ is the pullback $f^*$ of $\cO$-modules together with the pullback connection). How much do this property and Theorem \ref{th:extension} extend to $Y$-meromorphic flat bundles formally modelled on $\cN_Y$?

\begin{notation}\label{nota:connect}
Let $y_o\in Y$ and set $t_o=f(y_o)$, contained in the stratum $S_o$ of $\Delta$. Given an open connected neighbourhood $U$ of $t_o$ in $T$, we denote by $f^{-1}(U)^{y_o}$ the connected component of $f^{-1}(U)$ containing~$y_o$.
\end{notation}

\begin{corollaire}\label{cor:extensionf}
With the previous notation, if $U\subset T$ satisfies the assumptions as in Theorem \ref{th:extension}, the restriction $\gamma_{y_o}^+$ (\resp the pullback $f^+$) induces a bijection between the set of isomorphism classes of pairs $(\cM_{f^{-1}(U)^{y_o}},\iso_{\wh{f^{-1}(U)^{y_o}}})$ (\ie defined on $f^{-1}(U)^{y_o}$) and that of pairs $(\cM^{y_o},\iso_{\wh0})$ (\resp with that of pairs $(\cM_{U},\iso_{\wh U})$).
\end{corollaire}

The proofs of Theorem \ref{th:extension} and Corollaries \ref{cor:extension} and \ref{cor:extensionf} are given in Section \ref{sec:proofs}, where we also interpret them in terms of the sheaf of Stokes torsors (Corollary \ref{cor:Stcomp}). The notion of very good formal decomposition, as well as the main properties we use, are developed in Section~\ref{sec:verygood}. In Section~\ref{sec:isomono}, we give an application to deformations of a connection in the Birkhoff normal form and to the construction of Frobenius manifolds, in the spirit of \cite{C-D-G17a,C-D-G17b,C-G17b,C-G17a}.

\subsubsection*{Acknowledgements}
This work has benefitted from fruitful discussions with Giordano Cotti and Jean-Baptiste Teyssier. In particular, the uniqueness results of Section \ref{subsec:uniqueness}, originally obtained in some special cases, are due in this generality to Jean-Baptiste Teyssier in \cite{Teyssier17}, in an even stronger form. We thank the anonymous referee for useful suggestions for improving the presentation of the article.

\section{General results on the notion of very good formal decomposition}\label{sec:verygood}

The objects considered in this section are meromorphic bundles on a complex manifold, with poles along a divisor having normal crossing, and endowed with a flat meromorphic connection with poles at most along the same divisor. Moreover, we assume that they are isomorphic, in the formal neighbourhood of the divisor, to a simpler model, which we fix, and we wish to classify pairs consisting of a meromorphic flat bundle and a formal isomorphism with this model, up to isomorphism. When the model is \emph{good}, then in the neighbourhood of a point on the divisor, such pairs are uniquely determined by their restriction to a slice that is transversal to the natural stratum of the divisor passing through this point (Proposition \ref{prop:locconstcodimtwo}). We also prove a global uniqueness result (Corollary \ref{cor:uniquenessglobal}) obtained differently by J.-B.\,Teyssier: if the divisor is connected, then a pair is uniquely determined by its restriction to a generic smooth curve transverse to the divisor at a smooth point. This result is useful when combined with the Kedlaya-Mochizuki theorem, as it has consequences when the model is not good (Corollary \ref{cor:uniquenessglobalI}). The techniques that we use mix a variant of the Malgrange-Sibuya theorem and recent results by T.\,Mochizuki on Stokes-filtered local systems on the oriented real blow-up along the divisor.

\subsection{Setting and notation}\label{subsec:settingnotation}
Throughout this section, $X$ denotes a germ of complex analytic manifold of dimension $m$ along a \emph{connected} reduced divisor $D\subset X$ with normal crossings, whose components are denoted by $D_{i\in I}$. We~will assume, for the sake of simplicity, that these components are \emph{smooth}. We~denote by $\cO_X(*D)$ the sheaf of meromorphic functions with poles on $D$ at most. Let $U$ be an open set in $X$ and $x_o\in U$. We~denote by $D_1,\dots,D_\ell$ the components of~$D$ passing through $x_o$, and we choose local coordinates $(x_1,\dots,x_m)$ such that $D_i=\{x_i=0\}$. The divisor $D$ has a natural stratification by locally closed smooth connected complex submanifolds and, for $x_o\in D$, we denote by $D^{(x_o)}$ the stratum passing through $x_o$ and by $D_{x_o}$ the germ of $D$ at $x_o$. With the above notation, it is equal to the connected component of $(\bigcap_{i=1,\dots,\ell}D_i)\moins(\bigcup_{j\neq1,\dots,\ell}D_j)$ containing $x_o$.

\subsection{Good sets of polar parts}

The exponential behaviour of horizontal holomorphic sections of meromorphic connections is governed by exponential factors that are polar parts of meromorphic functions along $D$, that is, sections of the sheaf $\cO_X(*D)/\cO_X$. In higher dimensions, the asymptotic behaviour of such exponential factors can be complicated (the geometry of the sectors on which they have rapid decay or exponential growth can be complicated, even if $D$ has normal crossings), and it is useful to select a class of such polar parts for which the sectors have a simple geometry. This leads to the goodness property.

\begin{definition}[Goodness]\label{def:goodness}\mbox{}
\begin{enumerate}
\item\label{def:goodness1}
A nonzero germ $\varphi\in\cO_{X,x_o}(*D)/\cO_{X,x_o}$ is \emph{purely monomial} if some (or any) local representative in $\cO_{X,x_o}(*D)$ can be written as $u(x)/x^\bmm$, where $\bmm\in\NN^\ell\moins\{0\}$, and $u(x)$ is holomorphic and nonvanishing at $x_o$.
\item\label{def:goodness1b}
For $\varphi\in\Gamma(U,\cO_X(*D)/\cO_X)$, we say that $\varphi$ is \emph{good} if its germ at any $x_o\in U$ is purely monomial.
\item\label{def:goodness2}
A finite set $\Phi$ of polar parts $\varphi\in\Gamma(U,\cO_X(*D)/\cO_X)$ is \emph{good} if, for every pair \hbox{$(\varphi,\psi)\in\Phi^2$} with $\varphi\neq\psi$, the difference $\varphi-\psi$ is good.
\end{enumerate}
\end{definition}

\begin{convention}\label{conv:Phi}
For a finite subset $\Phi_{x_o}\subset\cO_{X,x_o}(*D)/\cO_{X,x_o}$, there exists a fundamental basis of Stein open neighbourhoods $V$ of $x_o$ such that each element of $\Phi_{x_o}$ is the germ at $x_o$ of a unique element of $\Gamma(V,\cO_X(*D))/\Gamma(V,\cO_X)$. We abuse the notation by considering $\Phi_{x_o}$ as a subset of the latter quotient, and also by denoting with the same letter an element of this subset and any of its lifts in $\Gamma(V,\cO_X(*D))$. For $\varphi\in\Phi_{x_o}$, it is then meaningful to say $\varphi$ vanishes along some connected component of the smooth part of $D$, or that $\varphi$ has no pole along some component of $D$. We can also consider the germ of $\varphi$ at any $x\in D$ in some neighbourhood of $x_o$.
\end{convention}

Let $x_o\in D$ and let $\Phi_{x_o}$ be a good (finite) set in $\cO_{X,x_o}(*D)/\cO_{X,x_o}$. We denote by $D(\Phi_{x_o})$ the union of (germs of) components of $D$ at $x_o$ along which at least some nonzero difference $\varphi-\psi$, with $\varphi,\psi\in\Phi_{x_o}$, has a pole.

Note that if $\Phi_{x_o}$ is good then for any $\varphi_o\in\Phi_{x_o}$, the set $\Phi_{x_o}-\varphi_o$ is also good, and, moreover, contains~$0$.

Let us fix $\varphi_o\in\Phi_{x_o}$. Recall that goodness then implies that the pole divisors of the elements $\varphi-\varphi_o$ in $\Phi_{x_o}-\varphi_o$ are totally ordered. Let us make this explicit in coordinates. By goodness, any $\varphi\!-\!\varphi_o\!\in\!(\Phi_{x_o}\!-\!\varphi_o)\moins\{0\}$ can be written as\vspace*{-3pt}
\begin{equation}\label{eq:mvarphi}
\varphi-\varphi_o=u_\varphi\cdot\bmx^{-\bmm_\varphi}:=u_\varphi x_1^{-m_1}\cdots x_\ell^{-m_\ell},
\end{equation}
with $\bmm_\varphi\in\NN^\ell\moins\{0\}$ and $u_\varphi$ is an invertible holomorphic function. The goodness condition implies that the set $\{\bmm_\varphi\mid\varphi\in\Phi_{x_o}\moins\{\varphi_o\}\}$ is totally ordered with respect to the partial order on $\NN^\ell$.

If there exist $\varphi\neq\psi$ in $\Phi_{x_o}$ such that $\varphi-\psi$ has no pole along some component of~$D$ at $x_o$, it will be convenient to consider the \emph{level structure} of $\Phi_{x_o}$ as defined now. Although the family $(\bmm_\varphi)_{\varphi\in\Phi_{x_o}}$ depends on the choice of the base point~$\varphi_o$, its maximum~$\bmm_o$ does not. For any $\varphi\in\Phi_{x_o}$, using the notation of \eqref{eq:mvarphi} we~set
\begin{equation}\label{eq:CPhi}
c(\varphi,\varphi_o)=
\begin{cases}
u_\varphi(x_o)&\hspace*{-2mm}\text{if }\bmm_\varphi=\bmm_o,\\
0&\hspace*{-2mm}\text{if }\bmm_\varphi<\bmm_o,
\end{cases}
\quad\text{and}\quad C(\varphi_o)=\{c(\varphi,\varphi_o)\mid\varphi\in\Phi_{x_o}\}\subset\CC.
\end{equation}

\begin{lemme}\label{lem:Cgeq2}
If $\#\Phi_{x_o}\geq2$, there exists $\varphi_o\in\Phi_{x_o}$ such that $\#C(\varphi_o)\geq2$.
\end{lemme}

\begin{proof}
If for the chosen $\varphi_o$ we have $\#C(\varphi_o)=1$, then for any $\varphi'_o\in\Phi_{x_o}$ such that $\bmm_{\varphi'_o}=\bmm_o$, we have $\#C(\varphi'_o)=2$. Indeed, we then have $C(\varphi'_o)=-C(\varphi_o)\cup\{0\}$.
\end{proof}

Assume $\#C(\varphi_o)\geq2$. We then obtain a nontrivial decomposition
\begin{equation}\label{eq:leveldec}
\Phi_{x_o}=\bigsqcup_{c\in C(\varphi_o)}\Phi_{x_o}(\varphi_o,c),\quad\Phi_{x_o}(\varphi_o,c)=\{\varphi\in\Phi_{x_o}\mid c(\varphi,\varphi_o)=c\}.
\end{equation}
Let $\bmm'_o$ denote the submaximal value of the sequence $(\bmm_\varphi)_{\varphi\in\Phi_{x_o}}$ (it may depend on the choice of $\varphi_o$). Then $\Phi_{x_o}(\varphi_o,c)$ is the inverse image in $\Phi_{x_o}-\varphi_o$ of $c/x^{\bmm_o}$ by the map induced by $\cO_{X,x_o}(*D)/\cO_{X,x_o}\to\cO_{X,x_o}(*D)/x^{-\bmm'_o}\cO_{X,x_o}$.

\begin{definition}[Level decomposition (first step)]\label{def:leveldec}
The decomposition \eqref{eq:leveldec} is called the first step of the level decomposition of $\Phi_{x_o}$ with base point $\varphi_o$.
\end{definition}

Every $\Phi_{x_o}(\varphi_o,c)$ is good, so that we can perform the same construction to it and get the complete level decomposition, which we will not define, since we will argue by induction only step by step.

\begin{remarque}
The above notions can be defined similarly along the stratum $D^{(x_o)}$, and then they restrict to the previous ones at $x_o$ (or at any point of the stratum). We then use the notation $C(x_o,\varphi_o)$ and $\Phi(x_o,\varphi_o,c)$.
\end{remarque}

\subsection{Classes of $D$-meromorphic flat bundles}\label{subsec:Nmarked}

Let $\cI_D$ be the reduced ideal of~$D$. We will set $X^*\!=\!X\!\moins\!D$ and, for any subset $J\!\subset\!I$, $D_J:=\bigcap_{i\in J}D_i$ and $D_J^\circ:=D_J\moins\bigcup_{i\notin J}D_i$. We denote by $\cO_{\wh D}$ the formal completion $\varprojlim_k\cO_X/\cI_D^k$, which we regard in an obvious way as a sheaf on $D$. Recall (\cf\eg\cite[Lem.\,I.1.1.13]{Bibi97}) that a section $f$ of $\cO_{\wh D}$ at a point $x_o\in D$ where the components of $D$ are $D_1,\dots,D_\ell$ consists of the data $(f_i)_{i\in\{1,\dots,\ell\}}$ of sections $f_i$ of $\cO_{\wh D_i}$ such that $f_i$ and $f_j$ coincide on $\cO_{\wh{D_i\cap D_j}}$ for all pairs $i,j=1,\dots,\ell$. In particular, $\cO_{\wh D}$ is naturally endowed with a differential~$\rd$, extending $\rd$ on $\cO_{X|D}$.

For an $\cO_X$-module~$\cM$, we denote by $\cM_{|D}$ its sheaf-theoretic restriction to $D$ and we set $\cM_{\wh D}:=\cO_{\wh D}\otimes_{\cO_{X|D}}\cM_{|D}$. By a \emph{$D$-meromorphic flat bundle} we mean a locally free $\cO_X(*D)$-module of finite rank endowed with an integrable connection. For such an~$\cM$, we say that $\cM_{\wh D}$ is a $D$-meromorphic formal flat bundle: it is $\cO_{\wh D}(*D)$-locally free of finite rank with an integrable connection.

Let us fix a $D$-meromorphic flat bundle $\cN$ on $X$.

\begin{definition}
Let $\cM$ be a coherent $\cO_X(*D)$-module an with integrable connection. We say that \emph{$\cM$ has $\cN$ as a $D$-formal model} if there exists an isomorphism \hbox{$\iso_{\wh D}:\cM_{\wh D}\isom\cN_{\wh D}$}.
\end{definition}

Note that $\cN$ has $\cN$ as a $D$-formal model when equipped with $\id:\cN_{\wh D}\isom\cN_{\wh D}$. If $\cM$ has $\cN$ as a $D$-formal model, then $\cM$ is also a $D$-meromorphic flat bundle, \ie it is $\cO_X(*D)$-locally free of finite rank. This justifies the terminology of \cite{Teyssier16} that $(\cM,\iso_{\wh D})$ is an \emph{$\cN$-marked $D$-meromorphic flat bundle}.

We define the category of $\cN$-marked $D$-meromorphic flat bundles in an obvious way: a morphism $\lambda:(\cM,\iso_{\wh D})\to(\cM',\iso_{\wh D}')$ is a morphism $\cM\to\cM'$ such that $\iso_{\wh D}'\circ\lambda_{\wh D}=\iso_{\wh D}$. If $(\cM,\iso_{\wh D})$ and $(\cM',\iso_{\wh D}')$ are isomorphic by an isomorphism~$\iota$, then such an isomorphism $\iota$ is unique, since $\iota_{\wh D}$ is uniquely determined. When $(X,D)=(\CC,0)$, the interest of considering such pairs has been emphasized by B.\,Malgrange in \cite{Malgrange83bb}.

For any open set $U\subset D$, we denote by $\cH(U,\cN)$ the set of isomorphism classes of pairs $(\cM_U,\iso_{\wh D})$ consisting of a (germ of) $D$-meromorphic flat bundle $\cM_U$ on some open neighbourhood of~$U$ in $X$ and a formal isomorphism $\iso_{\wh D}$ with $\cN_{\wh D}$ on~$U$. Owing to the uniqueness of isomorphisms between such pairs, we deduce that the presheaf $U\mto\cH(U,\cN)$ is a sheaf of sets $\cH_D(\cN)$ with a given section $(\cN,\id)$, and sections on $D$ of this sheaf consist of pairs $(\cM,\iso_{\wh D})$ as above (up to unique isomorphism).

The basic operations we will use are the complex blowing-ups, or sequences of such, and more generally proper modifications. A proper modification of $X$ is a proper morphism $X'\to X$ of complex manifolds inducing an isomorphism between open dense subsets of these manifolds. Since we consider pairs $(X,D)$ of manifolds endowed with a normal crossing divisor, we extend this notion as follows: a proper modification $e:(X',D')\to(X,D)$ is a proper modification $\modif:X'\to X$ such that $D'$ is contained $\modif^{-1}(D)$ and which induces an isomorphism $X\moins \modif^{-1}(D)\to X\moins D$. When $D'=\modif^{-1}(D)$, we can compare the sheaves $\cH_D(\cN)$ and $\cH_{D'}(\modif^+\cN)$.

\begin{lemme}\label{lem:modifHD}
Let $\modif:(X',D')\to(X,D)$ be a proper modification such that $D'=\modif^{-1}(D)$. Then $\modif^+:\cH_D(\cN)\to\modif_*\cH_{D'}(\modif^+\cN)$ is an isomorphism, having $\modif_+$ as its inverse.
\end{lemme}

\begin{proof}
It is known that $(\modif^+,\modif_+)$ (where $\modif_+$ is the pushforward $\modif_*$ of $\cO_{X'}(*D')$-modules with the pushforward connection) forms a pair of quasi-inverse functors between the categories of $D$- (\resp $D'$-) meromorphic flat bundles on $X$ (\resp $X'$). On the other hand, by right exactness of the tensor product, we have $\cO_{\wh D'}=\modif^*\cO_{\wh D}$. We thus have a functorial isomorphism of $\cO_{\wh D}$-modules, compatible with connections
\[
\modif_*(\cO_{\wh D'}\otimes\cM')\simeq R^0\modif_*(\modif^*\cO_{\wh D}\otimes\modif^*(\modif_*\cM'))\simeq\cH^0(\cO_{\wh D}\otimes \bR\modif_*\modif^*(\modif_*\cM')\simeq\cO_{\wh D}\otimes(\modif_*\cM').
\]
The result follows.
\end{proof}

Recall (\cf\cite[Th.\,5.5]{Malgrange04}) that, if a closed analytic subset $S\subset D$ has codimension $\geq2$ in $D$, then the restriction functor, from the category of $D$-meromorphic flat bundles on $X$ to that of $(D\moins S)$-meromorphic flat bundles on $X\moins S$, is an equivalence of categories. The next proposition shows that a~similar result holds for $\cN$-marked $D$-meromorphic flat bundles.

\begin{proposition}\label{prop:Malgrangeextension}
The restriction functor, from the category of $\cN$-marked $D$\nobreakdash-meromor\-phic flat bundles on $X$ to that of $\cN$-marked $(D\moins S)$-meromorphic flat bundles on $X\moins S$, is an equivalence of categories.
\end{proposition}

\begin{proof}
By the result of B.\,Malgrange aforementioned, we are reduced to proving essential surjectivity, that is, given $D$-meromorphic flat bundles $\cM,\cN$, any formal isomorphism $\iso_{\wh{D\moins S}}:\cM_{\wh{D\moins S}}\simeq\cN_{\wh{D\moins S}}$ extends as an isomorphism \hbox{$\iso_{\wh D}:\cM_{\wh D} \simeq\cN_{\wh D}$}. This is a local question on $S$, and we can assume that $\cM$ and $\cN$ are $\cO_X(*D)$-free. The components of $\iso_{\wh{D\moins S}}$ on bases of $\cM,\cN$ are sections of $\cO_{\wh{D\moins S}}(*(D\moins S))$. By considering the polar coefficients, we are led to showing that, if $j:D\moins S\hto D$ denotes the inclusion, the natural monomorphism $\cO_{\wh D}\to j_*\cO_{\wh{D\moins S}}$ is an isomorphism.

Let $f=(f_i)$ be a section of $\cO_{\wh{D\moins S}}$ on $\nb_D(x_o)\moins S$ for some $x_o\in S$. We simply write $D=\nb(x_o)$. Since $S$ has codimension $\geq2$ in each $D_i$, any section $f_i$ of $\cO_{\wh{D_i\moins S}}$ extends in a unique way as a section of $\cO_{\wh{D_i}}$, as seen by applying Hartogs's theorem to the coefficients $f_{i,k}\in\cO(D_i\moins S)$ of the formal series $f_i=\sum_{k\geq0}f_{i,k}x_i^k$. At a point of $D_i\cap D_j\cap S$, the coefficients $f_{i,k,\ell},f_{j,k,\ell}\in\cO((D_i\cap D_j)\moins S)$ of $f_i,f_j$ on the monomial $x_i^kx_j^\ell$ coincide on their domains, and since $D_i\cap D_j\cap S$ has codimension $\geq1$ in $D_i\cap D_j$, they coincide everywhere on $D_i\cap D_j$. As a consequence, the section $f$ of $\cO_{\wh{D\moins S}}$ extends (in a unique way) as a section of $\cO_{\wh D}$ on $\nb_D(x_o)$, as wanted.
\end{proof}

On the other hand, extension of morphisms can be done in codimension one with respect to $D$, according to Hartogs's theorem. Assume for example that $D$ is smooth, and let $H\subset D$ be a codimension-one closed analytic subset. Let us denote by $j:U=D\moins H\hto D$ the open inclusion. For a $D$-meromorphic flat bundle $\cM$ on $X$, we denote by $\cM_{|U}$ its sheaf-theoretic restriction to $U\subset X$.

\begin{lemme}\label{lem:Hartogs}
Let $\cM,\cM'$ be $D$-meromorphic flat bundles on $X$. Then, under the~above assumptions, any morphism $f_U:\cM_{|U}\to\cM'_{|U}$ extends in a unique way as a morphism $f:\cM\to\cM'$. If $(\cM,\iso_{\wh D}),(\cM',\iso'_{\wh D})$ are $\cN$-marked $D$-meromorphic flat bundles, any morphism $f_U:(\cM,\iso_{\wh D})_{|U}\to(\cM',\iso'_{\wh D})_{|U}$ extends in a unique way as a morphism $f$ between these pairs. In both cases, if $f_U$ is an isomorphism, then so is~$f$.
\end{lemme}

\begin{proof}
By the uniqueness assertion, the question is local on $D$ and since $D$ is smooth, we can assume that $X=D\times(\CC,0)$ with $D$ simply connected. Since $\cM,\cM'$ are flat bundles on $X\moins D$ and $\pi_1(X\moins D)=\pi_1(\CC\moins\{0\})$, we can assume that $X=D\times B_\epsilon$, where $B_\epsilon$ is the open disc of radius $\epsilon>0$ centered at the origin in $\CC$. By horizontality, $f_U$ is defined on $U\times B_\epsilon$.

For the same reason, the sheaf $\cHom^\nabla(\cM,\cM')_{X\moins D}$ is a locally constant sheaf on $X\moins D=D\times B^*_\epsilon$ ($B^*_\epsilon=B_\epsilon\moins\{0\}$). It follows that any section of this sheaf on $\{x\}\times B^*_\epsilon$ uniquely extends as a global section. In particular $f_U$ uniquely extends to $X\moins D$.

We now apply Hartogs theorem. Let us fix bases of $\cM,\cM'$ as $\cO_X(*D)$-modules (recall that we work locally on $D$). We have obtained a morphism $\cM_{X\moins H}\to\cM'_{X\moins H}$ extending $f_U$. In the chosen bases, the entries of the matrix of this morphism are holomorphic functions on $X\moins H$. Since $H$ has codimension two in $X$, they extend (in a unique way) as holomorphic functions on $X$, hence the first statement of the lemma. The other statements are then straightforward.
\end{proof}

\begin{corollaire}\label{cor:Hartogs}
Under the above assumptions, the natural morphism $\cH_D(\cN)\to j_*j^{-1}\cH_D(\cN)$ is injective.
\end{corollaire}

\begin{proof}
The question is local at points of $H$ and we can assume that $D$ is a small open neighbourhood of such a point. We are thus reduced to proving that if two pairs $(\cM,\iso_{\wh D}),(\cM',\iso'_{\wh D})$ are isomorphic on $D\moins H$, they are isomorphic. But by the previous lemma, the isomorphism on $D\moins H$ lifts in a unique way as an isomorphism on $D$.
\end{proof}

\subsection{Very good formal decomposition}
By a \emph{good decomposable $D$-meromorphic flat bundle} we mean a locally free $\cO_X(*D)$-module $\cM^\good$ with integrable connection $\nabla^\good$ (that we usually omit to mention), which is globally (on~$X$) isomorphic as such to a direct sum
\begin{equation}\label{eq:gooddecbdle}
\cM^\good\simeq\bigoplus_{\varphi\in\Phi}(\cE^\varphi\otimes\cR_\varphi),
\end{equation}
where
\begin{itemize}
\item
$\Phi$ is a finite subset of $\Gamma(X,\cO_X(*D)/\cO_X)$ which is \emph{good} at every point $x_o\in D$,
\item
$\cR_\varphi$ is a $D$-meromorphic flat bundle having regular singularities along $D$,
\item
$\cE^\varphi:=(\cO_X(*D),\rd+\rd\varphi)$.\footnote{This definition is local on $X$, since $\varphi$ lifts to $\wt\varphi\in\Gamma(X,\cO_X(*D))$ if $X$ is Stein. Otherwise, on a Stein open covering $(U_i)$ of $X$, various liftings $\wt\varphi_i$ give rise to the cocycle $(\exp(\wt\varphi_i-\wt\varphi_j))_{ij}\in H^1(X,\cO_X^*)$ defining a rank-one vector bundle $\cL_\varphi$, and $\rd+\rd\varphi$ is a well-defined connection on $\cL_\varphi(*D)$.} 
\end{itemize}

\begin{definition}[Very good formal decomposition, \cite{Bibi93}]
Let $\cM$ be a $D$-meromorphic flat bundle and let $\cM^\good$ be a good decomposable $D$-meromorphic flat bundle. If $\cM$ has $\cM^\good$ as a $D$-formal model, we say that~$\cM$ has a \emph{very good formal decomposition} along $D$ at each point of $D$.
\end{definition}

If $\cM$ has $\cM^\good$ $D$-formal model, we can fix a $D$-formal isomorphism $\iso_{\wh D}:\cM_{\wh D}\isom\cM^\good_{\wh D}$. We thus have an $\cM^\good$-marked $D$-meromorphic flat bundle $(\cM,\iso_{\wh D})$, which defines a section on $D$ of the sheaf $\cH_D(\cM^\good)$ (\cf Section \ref{subsec:Nmarked}).

Let $\gamma:(\CC^\ell,0)\to(X,D)$ be a germ of holomorphic map such that $\gamma^{-1}(D)$ is a normal crossing divisor. Then $\gamma^+\cM^\good$ is a good model (but various $\gamma^*\varphi$ may coincide, leading to grouping the corresponding $\gamma^+\cR_\varphi$) and $\gamma^+(\cM,\iso_{\wh D}):=(\gamma^+\cM,\gamma^*\iso_{\wh D})$ is a $\gamma^+\cM^\good$-marked $\gamma^{-1}(D)$-meromorphic flat bundle on $(\CC^\ell,\gamma^{-1}(D))$ near $0\in\CC^\ell$.

By definition, $\cM^\good$ is endowed with the marking $\id$. Any other marking is an automorphism $\iso_{\wh D}:\cM^\good_{\wh D}\isom\cM^\good_{\wh D}$.

\begin{proposition}\label{prop:blockdiag}
Any such automorphism is block diagonal with respect to the decomposition \eqref{eq:gooddecbdle}, with the $(\varphi,\varphi)$-block being induced by an automorphism of $\cR_\varphi$.
\end{proposition}

\begin{proof}[Sketch of proof]
Let $x_o\in D$. One computes that, if $\varphi_{x_o}\neq\psi_{x_o}$ in $\Phi_{x_o}$, then
\[
\cHom^\nabla((\cE^\varphi\otimes\cR_\varphi)_{\wh D,x_o},(\cE^\psi\otimes\cR_\psi)_{\wh D,x_o})=0,
\]
being formed of $\nabla$-horizontal sections of $\cE^{\psi-\varphi}\otimes\cHom(\cR_{\varphi,\wh D,x_o},\cR_{\psi,\wh D,x_o})$. On the other hand, if $\varphi_{x_o}=\psi_{x_o}$ in $\Phi_{x_o}$, we have by regularity,
\begin{align*}
\cHom^\nabla((\cE^\varphi\otimes\cR_\varphi)_{\wh D,x_o},(\cE^\psi\otimes\cR_\psi)_{\wh D,x_o})&=\cHom^\nabla(\cR_{\varphi,\wh D,x_o},\cR_{\psi,\wh D,x_o})\\
&\simeq\cHom^\nabla(\cR_\varphi,\cR_\psi)_{|D,x_o}.
\end{align*}
Let $\varphi\neq\psi$ in $\Phi$ and let $D(\varphi,\psi)$ be the (nonempty) union of components of $D$ where $\varphi\neq\psi$. Let $D'$ be the union of the other components. Since $D$ is assumed to be connected, each connected component of $D'$ cuts $D(\varphi,\psi)$. The above computation shows that the block $\iso_{\varphi,\psi,\wh D}$ is a section of a locally constant sheaf on $W'\moins D'$, where~$W'$ is a neighbourhood of $D'$ in $X\moins D(\varphi,\psi)$. Since each connected component of this neighbourhood has a limit point in $D(\varphi,\psi)$, the section must vanish near this limit point. Hence it is zero.
\end{proof}

\begin{corollaire}
Let $\iso_{\wh D}$ and $\iso'_{\wh D}$ be two markings of $\cM$ with model $\cM^\good$. Let $x_o\in D$, let $\ell=\codim_XD^{(x_o)}$ and let $\gamma:(\CC^\ell,0)\to(X,D)$ be a transversal slice of~$D^{(x_o)}$ at $x_o=\gamma(0)$. If $\gamma^*\iso_{\wh D}=\gamma^*\iso'_{\wh D}$, then $\iso_{\wh D}=\iso'_{\wh D}$.
\end{corollaire}

\begin{proof}
It is enough to prove the assertion for markings $\id$ and $\iso_{\wh D}$ of $\cM^\good$. We can work on a connected open neighbourhood of $D$ in $X$, which we still denote by $X$, so that we can assume that $X\moins D$ is connected. Since, by Proposition \ref{prop:blockdiag}, $\iso_{\wh D}$ is block diagonal with blocks $\iso_{\wh D,\varphi,\varphi}$, we are reduced to proving that if $\gamma^*\iso_{\wh D,\varphi,\varphi}=\id$, then $\iso_{\wh D,\varphi,\varphi}=\id$. This is a consequence of the fact that $\iso_{\wh D,\varphi,\varphi}$ is a global section of a locally constant sheaf on a connected set, as shown in the same proposition.
\end{proof}

\subsection{Good versus very good formal decomposition}
More general is the notion of \emph{good formal decomposition} (\cf\cite{Bibi97, Mochizuki07b,Kedlaya09,Mochizuki08,Kedlaya10}). We~will make explicit the differences between the two notions. We denote by $\cO_{\wh{D^{(x_o)}}}$ the sheaf locally defined as $\varprojlim_k\cO_X/(x_1,\dots,x_\ell)^k$ (\ie the formalization of $\cO_X$ along the stratum of $x_o$) and we say that $\cM$ has a good decomposition at $x_o$ with formal model $\cM^\good$ if there exists an isomorphism in the neighbourhood of $x_o$:
\begin{equation}\label{eq:gooddec}
\cM_{\wh{D^{(x_o)}}}:=\cO_{\wh{D^{(x_o)}}}\otimes\cM\isom\cO_{\wh{D^{(x_o)}}}\otimes\cM^\good=:\cM^\good_{\wh{D^{(x_o)}}}.
\end{equation}

\begin{remarque}
Let us make clear that, starting from any $D$-meromorphic flat bundle, one can find a sequence of blowing-ups (locally on $X$ in the complex analytic setting, \cf \cite{Kedlaya09,Kedlaya10}, and globally in the projective setting, \cf \cite{Mochizuki07b,Mochizuki08}) so that, after local ramifications on the blown-up space giving rise to a space denoted by $(X',D')$, the pullback of the $D$-meromorphic flat bundle, which is now a $D'$-meromorphic flat bundle, admits a good formal decomposition, \ie for each stratum~$D'_I$ of the natural stratification of $D'$, when tensored with the formal completion $\cO_{\wh{D'_I}}$ along this stratum, it admits an isomorphism with a good model. In general, such an isomorphism cannot be lifted as an isomorphism formally along the divisor~$D'$, \ie by tensoring instead with $\cO_{\wh{D'}}$.
\end{remarque}

\begin{exemple}\label{ex:goodverygood}
However (\cf \cite[Th.\,I.2.2.4]{Bibi97} in dimension two, and in general \cite[Cor.\,11.28]{Bibi10} which is a consequence of results by T.\,Mochizuki \cite[\S2.4.3]{Mochizuki08} on good lattices), if we assume moreover that, given $x_o\in D$, for any pair $\varphi\neq\psi\in\Phi_{x_o}$ the difference $\varphi-\psi$ has poles along \emph{all} components of $D_{x_o}$, then any good formal decomposition along $D^{(x_o)}$ is very good.
\end{exemple}

\begin{exemple}[\Cf{\cite[Lem.\,I.2.2.3]{Bibi97}}]\label{exem:decomp}
On the other hand, let $x_o\in D$ and assume that there exists a component $D_i$ of $D_{x_o}$ along which all nonzero differences $\varphi-\psi$ ($\varphi,\psi\in\Phi_{x_o}$) vanish (\ie their representatives do not have a pole, in other words, $D(\Phi_{x_o})\neq D_{x_o}$. Then the germ $\cH_D(\cM^\good)_{x_o}$ reduces to $(\cM^\good,\id)$, that is, for any germ $(\cM,\iso_{\wh D})_{x_o}$, there exists a (unique) lifting $\iso_{x_o}:\cM_{x_o}\isom\cM^\good_{x_o}$. Indeed (\cf below), the sheaf $\cAut^\rdD(\cM^\good_{\partial\wt X})_{|\varpi^{-1}(x_o)}$ is equal to~$\id$, so $\St_D(\cM^\good)_{x_o}$ is also reduced to $\id$, and we can apply Theorem \ref{th:HSt} below.
\end{exemple}

Checking whether a good decomposition is very good can be done inductively with respect to the level decomposition of $\cM_{\wh D}$, which we define now, in a way parallel to the level decomposition of $\Phi_{x_o}$.

\begin{proposition}[First step of the level decomposition]\label{prop:devissage}
Assume that $\cM$ is a (germ at~$x_o$ of a) good $D$\nobreakdash-meromorphic flat bundle with good formal model \eqref{eq:gooddecbdle} (\ie \eqref{eq:gooddec} holds) and that $D(\Phi_{x_o})= D_{x_o}$. Then, for every $c\in C(x_o,\varphi_o)$ (\cf \eqref{eq:CPhi}) there exists a $D$\nobreakdash-meromorphic flat bundle~$\cM_c$ in the neighbourhood of $D^{(x_o)}$ satisfying the following properties with respect to \eqref{eq:gooddecbdle}:
\begin{align*}\tag{\ref{prop:devissage}$\,*$}\label{eq:devissage*}
\cM_{\wh D|D^{(x_o)}}&\simeq\bigoplus_{c\in C(x_o,\varphi_o)}\cM_{c,\wh D|D^{(x_o)}},\\
\tag{\ref{prop:devissage}$\,**$}\label{eq:devissage**}
\cM_{c,\wh{D^{(x_o)}}}&\simeq\bigoplus_{\eta\in\Phi(x_o,\varphi_o,c)}(\cE^{\eta}\otimes\cR_\eta)_{\wh{D^{(x_o)}}}\quad\forall c\in C(x_o,\varphi_o).
\end{align*}
Moreover, we have $\cM_{\wh D|D^{(x_o)}}\!\simeq\!\cM_{\wh D|D^{(x_o)}}^\good$ if and only if the same property holds for each~$\cM_c$.
\end{proposition}

\begin{proof}
See \cite[p.\,189--190]{Bibi10} for the first part. For the second assertion, the~``if'' part follows from the first part. Conversely, assume $\cM_{\wh D|D^{(x_o)}}\!\simeq\!\cM_{\wh D|D^{(x_o)}}^\good$. It is enough to prove that this isomorphism is block diagonal with respect to the decomposition~\eqref{eq:devissage*}. This amounts to showing that there is no nonzero morphism $\cM_{c,\wh D|D^{(x_o)}}\to\cM^\good_{c',\wh D|D^{(x_o)}}$ if $c\neq c'$ and, by \eqref{eq:devissage**} and faithful flatness of $\cO_{\wh{D^{(x_o)}}}$ over $\cO_{\wh D|D^{(x_o)}}$, no nonzero morphism $(\cE^{\varphi}\otimes\cR_\varphi)_{\wh{D^{(x_o)}}}\to(\cE^{\psi}\otimes\cR_\psi)_{\wh{D^{(x_o)}}}$ if $c(\varphi,\varphi_o)\neq c(\psi,\varphi_o)$. This is implied by the vanishing of any horizontal section of $\cHom(\cR_{\varphi,\wh{D^{(x_o)}}},\cE^{\psi-\varphi}\otimes\cR_{\psi,\wh{D^{(x_o)}}})$ if $\varphi\neq\psi$ in $\Phi$ in the neighbourhood of $D^{(x_o)}$, a property that is standard.
\end{proof}

\subsection{The sheaf of Stokes torsors}\label{subsec:StT}\label{subsec:verygoodformal}

\subsubsection*{Reminder of the theory in dimension one}
The approach followed here intends to generalize in higher dimensions that of B.\,Malgrange \cite{Malgrange83bb} in dimension one, relying on the so-called Malgrange-Sibuya theorem (\cf also \cite{B-V89}), approach that we quickly recall here. We thus assume that $(X,D)=(\CC,0)$. Let $\wt\CC=S^1\times\RR_+$ be the real oriented blow-up of $\CC$ at the origin, that is, the space of polar coordinates. It is endowed with the sheaf $\cA_{\wt\CC}$ ($C^\infty$ functions on $\wt\CC$ satisfying the Cauchy-Riemann equation on~$\CC^*$, hence $\cA_{\wt\CC|\CC^*}=\cO_{\CC^*}$) and its subsheaf $\cA_{\wt\CC}^\rrd$ consisting of functions having rapid decay along $S^1\times\{0\}$. We now only consider the sheaf-theoretic restrictions $\cA_{S^1},\cA^\rrd_{S^1}$ of these sheaves to the boundary $S^1\times\{0\}=S^1$. For a model meromorphic flat bundle $\cM^\good$ on $(\CC,0)$, the sheaf $\cAut_{S^1}^\rrd(\cM^\good)$ consists of local automorphisms of $\cA_{S^1}\otimes\cM^\good$ compatible with the connection that are asymptotic to $\id$ on the open set of $S^1$ where they are defined. The set $H^1(S^1,\cAut_{S^1}^\rrd(\cM^\good))$, that is, the set of \emph{Stokes torsors}, classifies meromorphic flat bundles~$\cM$ endowed with a formal isomorphism $\iso_{\wh0}:\cM_{\wh0}\isom\cM^\good_{\wh0}$, according to \cite[Th.\,3.4]{Malgrange83bb}. It can be endowed with a richer structure (\cf \cite{B-V89}) that we will not consider here.

\begin{exemple}\label{exem:Stokes}
Let $t_{o,1},\dots,t_{o,n}$ be pairwise distinct complex numbers and let $\cR_i$ ($i=\nobreak1,\dots,n$) be a free $\cO_{\CC,0}(*0)$-module with a regular connection. Set
\[
\cM^\good=\bigoplus_{i=1}^n\cE^{-t_{o,i}/z}\otimes\cR_i.
\]
If $\theta_o\in S^1$ is general, we can reindex the numbers $t_{o,i}$ so that, for $i,j\in\{1,\dots,n\}$,
\[
i< j\iff\mathrm{Re}((t_{o,i}-t_{o,j})\mathrm{e}^{-\sfi\theta_o})<0,
\]
and an element of $H^1(S^1,\cAut_{S^1}^\rrd(\cM^\good))$ is a pair $(S^+,S^-)$ of matrices (the Stokes matrices), one being upper triangular, the other one being lower triangular, both having $\id$ as their diagonal part. Let us recall this correspondence. It also depends on a choice of a horizontal basis of $\bigoplus_i\cR_i$ in a small open sector centered at $\theta_o$ and another one in the opposite sector. The set $H^1$ is computed as \v Cech cohomology via the Leray covering consisting of the two intervals with boundary points $\theta_o,\theta_o+\pi$, slightly extended, so that it is identified with
\[
\Gamma\bigl((\theta_o-\epsilon,\theta_o+\epsilon),\cAut_{S^1}^\rrd(\cM^\good)\bigr)\times\Gamma\bigl((\theta_o+\pi-\epsilon,\theta_o+\pi+\epsilon),\cAut_{S^1}^\rrd(\cM^\good)\bigr),
\]
and we identify the first (\resp second) term with upper (\resp lower) triangular constant matrices $S^+$ (\resp $S^-$) with $\id$ on the diagonal since, for a constant matrix~$(S_{ij})$, the matrix $S_{ij}\exp(t_{o,i}-t_{o,j})/z$ has rapid decay in a small open sector centered at $\theta_o$ (\resp $\theta_o+\pi$) if and only if $i<j$ (\resp $i>j$). 

If we make the complex numbers $t_{o,i}$ vary as $t_i$, but remaining pairwise distinct, the space $H^1(S^1,\cAut_{S^1}^\rrd(\cM^\good))$ varies in a locally constant way, as we will recall below. Understanding the behaviour of the Stokes matrices in such a variation needs more care, since they depend on the choice of the generic $\theta_o\in S^1$, which can become non generic for some values of the parameters $t_i$. Apparent real singularities may thus appear in the parameter space. This explains why we will use the language of sheaves of Stokes torsors instead of that of Stokes matrices: we wish to avoid these apparent singularities. However, if we vary $t_i$ along a real parameter in such a way that $\theta_o$ and the corresponding order can be chosen constant all along the deformation, then the representation in terms of Stokes matrices holds all along the deformation. This includes limit cases where some $t_i$'s may coincide: coalescence of eigenvalues occurs in the sense of \cite{C-D-G17a}.
\end{exemple}

\subsubsection*{Real oriented blow-up}
In higher dimensions, the Stokes sectors are multi-sectors, which are conveniently defined on the real oriented blown-up space of~$X$ along the components of $D$, a space that has local coordinates given by the polar coordinates around each component of $D$.

Let us now consider the higher-dimensional situation of $(X,D)$ as above. We denote by $\varpi:\wt X=\wt X(D_{i\in I})\to X$ the real-oriented blowing-up of the components $D_{i\in I}$ in $X$ (\cf \eg \cite[\S8.2]{Bibi10} for the global setting). In a local coordinate system $(x_1,\dots,x_m)$ as above, we identify~$\wt X$ with $(S^1)^\ell\times(\RR_+)^\ell\times\CC^{m-\ell}$ (polar coordinates with respect to $x_1,\dots,x_\ell$). We set $\partial\wt X:=\varpi^{-1}(D)$, locally isomorphic to the product $(S^1)^\ell\times\partial(\RR_+)^\ell\times\CC^{m-\ell}$. On $\wt X$ we consider the sheaves $\cA_{\wt X}$ ($C^\infty$~functions on~$\wt X$ satisfying the Cauchy-Riemann equation on $X^*$) and $\cA^\rdD_{\wt X}$ (holomorphic functions on $X^*$ having rapid decay along~$\partial\wt X$). These sheaves coincide with $\cO_{X^*}$ on~$X^*$, so we will only consider their sheaf-theoretic restrictions to $\partial\wt X$, where we have a strict inclusion
\[
\cA^\rdD_{\partial\wt X}\subset\cA_{\partial\wt X}.
\]
On the other hand, setting
\[
\cA_{\wh{\wt X|D}}:=\varprojlim_k\bigl(\cA_{\partial\wt X}\bigm/\varpi^{-1}\cI_D^k\cdot\cA_{\partial\wt X}\bigr),
\]
we have an exact sequence
\[
0\to\cA^\rdD_{\partial\wt X}\to\cA_{\partial\wt X}\to\cA_{\wh{\wt X|D}}\to0.
\]
(See \cite[\S II.1.1]{Bibi97} and \cite{Mochizuki10}, and the references therein for details). For an $\cO_X(*D)$\nobreakdash-module~$\cM$, we denote by $\cM_{\partial\wt X}$ the $\cA_{\partial\wt X}$-module \hbox{$\cA_{\partial\wt X}\otimes_{\varpi^{-1}\cO_{X|D}}\varpi^{-1}\cM_{|D}$}, and similarly for $\cM^\rdD_{\partial\wt X}$. If, moreover, $\cM$ is endowed with an integrable connection \hbox{$\nabla:\cM\to\Omega^1_X\otimes\cM$}, then $\nabla$ lifts as an operator $\nabla:\cM_{\partial\wt X}\to\varpi^{-1}\Omega^1_X\otimes\cM_{\partial\wt X}$ that satisfies $\nabla^2=0$, and similarly for $\cM^\rdD_{\partial\wt X}$.

\subsubsection*{The sheaf of Stokes torsors}
Let $\cEnd^\nabla(\cM^\good_{\partial\wt X})$ be the sheaf of endomorphisms of $\cM^\good_{\partial\wt X}$ compatible with $\nabla^\good$. We denote by $\cAut^\rdD(\cM^\good_{\partial\wt X})$ its subsheaf consisting of sections whose image in $\cEnd(\cM^\good_{\wh{\wt XD}})$ is equal to~$\id$. It is a sheaf of groups on $\partial\wt X$. We consider the presheaf on $X$ defined by
\[
U\mto H^1\bigl(\varpi^{-1}(U),\cAut^\rdD(\cM^\good_{\partial\wt X})\bigr),
\]
and we denote by $\St_D(\cM^\good)$ the associated sheaf, which we call \emph{the sheaf of Stokes torsors}.\footnote{As pointed out to me by J.-B.\,Teyssier, this presheaf is already a sheaf, due to the theorem of Malgrange-Sibuya mentioned below.} This is a sheaf of pointed sets (pointed by the class of $\id$).\footnote{It could be given a richer structure as in \cite[Th.\,2.2]{Malgrange83dbb}, \cf\cite{Teyssier16}, but we will not need~it.}

We define below a morphism of presheaves
\begin{equation}\label{eq:HStpre}
\cH_D(U,\cM^\good)\to H^1\bigl(\varpi^{-1}(U),\cAut^\rdD(\cM^\good_{\partial\wt X})\bigr)
\end{equation}
and we then consider the associated morphism of sheaves of pointed sets
\begin{equation}\label{eq:HSt}
\cH_D(\cM^\good)\to\St_D(\cM^\good).
\end{equation}

By a theorem of H.\,Majima \cite[Th.\,III.2.1, p.\,121]{Majima84} (\cf also \cite[Th.\,(3.1)]{Bibi93} that is stated in dimension two, but the proof can be adapted to arbitrary dimensions, and \cite[Prop.\,20.1.1]{Mochizuki08}), for any $x_o\in D$ and $\theta_o\in\varpi^{-1}(x_o)$, the germ of $\iso_{\wh D}$ at~$x_o$ can be lifted as a germ of isomorphism
\[
\cM_{\partial\wt X,\theta_o}\isom\cM^\good_{\partial\wt X,\theta_o}.
\]
We can thus find a covering $\cU$ of $\varpi^{-1}(U)$ by open subsets where such a lifting exists, and by comparing the liftings on the intersection of two open subsets we obtain a cocycle in $Z^1\bigl(\cU,\cAut^\rdD(\cM^\good_{\partial\wt X})\bigr)$. Two families of liftings on $\cU$ define two cocycles related by the action of a coboundary, so the corresponding class in $H^1\bigl(\cU,\cAut^\rdD(\cM^\good_{\partial\wt X})\bigr)$ is independent of the choice of local liftings. Passing to the limit with respect to the coverings $\cU$ leads to the definition of the morphism \eqref{eq:HStpre} and hence to that of \eqref{eq:HSt}.

\begin{theoreme}\label{th:HSt}
The morphism of sheaves \eqref{eq:HSt} is an isomorphism.
\end{theoreme}

\begin{proof}
It is completely similar to that of \cite[Th.\,3.4]{Malgrange83bb}, as extended to the case where $D$ is smooth (\cf\cite[\S2]{Malgrange83dbb}, see also \cite[\S II.6.d]{Bibi97}), where one has to replace $S^1$ with $\varpi^{-1}(x_o)\simeq(S^1)^\ell$ for some $\ell$. The proof of injectivity needs no change, and the proof of surjectivity needs the reference to the (generalized) Malgrange-Sibuya theorem \cite[Th.\,12.2]{Bibi10}.
\end{proof}

\begin{corollaire}
Let $U\subset D$ be an open subset and $(\cM,\iso_{\wh D})_{|U}\in\Gamma(U,\cH_D(\cM^\good))$. The $D$-formal isomorphism $\iso_{\wh D}$ can be lifted (in a unique way) as an isomorphism $\iso_{U}:\cM_{|U}\isom\cM^\good_{|U}$ if and only if the image of $(\cM,\iso_{\wh D})_{|U}$ in $\Gamma(U,\St_D(\cM^\good))$ is equal to $\id$.\qed
\end{corollaire}

\subsection{Generic local constancy}
We keep the setting of Section \ref{subsec:verygoodformal}.

\begin{proposition}\label{prop:locconstgeneric}
The sheaf $\cH_D(\cM^\good)$ is a locally constant sheaf of pointed sets when restricted to the smooth open subset of $D$, and its fiber at a smooth point $x_o\in D$ is in bijection with $\cH_0(\gamma^+\cM^\good)$ for any germ $\gamma:(\CC,0)\to (X,x_o)$ transverse to $D$.
\end{proposition}

\begin{proof}[First proof]
It is known (\cite[Th.\,2.2]{Malgrange83dbb}, see also \cite[Th.\,II.6.1\,\&\,Cor.\,II.6.7]{Bibi00b}) that the sheaf $\St_D(\cM^\good)$ is a \emph{locally constant sheaf} of pointed sets when restricted to the smooth part of $D$ and that its sheaf-theoretic restriction to a germ of a smooth curve transversal to $D$ at a smooth point of $D$ is equal to the sheaf of Stokes torsors of the meromorphic flat bundle that is the restriction to $\cM^\good$ to this curve. We conclude with Theorem \ref{th:HSt}.
\end{proof}

\begin{proof}[Second proof]
We use the notion of Stokes structure and Stokes filtration as in \cite{Mochizuki10b} and \cite{Bibi10}. The Riemann-Hilbert correspondence in this setting induces bijective correspondence between the isomorphism classes (on $U\!\subset\!D^\smooth$) of meromorphic flat bundles formally isomorphic to $\cM^\good_{|U}$ and isomorphism classes of good Stokes-filtered local systems $(\cL,\cL_\bbullet)$ on $\varpi^{-1}(U)\!\subset\!\partial\wt X$ whose associated graded object is isomorphic to the Stokes-filtered local system $(\cL^\good,\cL^\good_\bbullet)$ attached to~$\cM^\good_{|U}$.

On the one hand, by applying the general result of \cite[Th.\,4.13]{Mochizuki10b}, which we will also use in a more general setting below (see also \cite[Appendix]{Bibi16b}), the sheaf on $D^\smooth$ classifying the Stokes-filtered local systems is a locally constant sheaf of sets compatible with the restriction to the curve $\gamma$. The supplementary choice of a $D$-formal isomorphism $\iso_{\wh D}$ is equivalent to the choice of an isomorphism $\iso_{\gr D}:(\gr\cL,\gr\cL_\bbullet)\isom(\cL^\good,\cL^\good_\bbullet)$, which is a thus section of a local system. It~follows that the sets of isomorphism classes of pairs $\bigl((\cL,\cL_\bbullet),\iso_{\gr D}\bigr)_{|U}$ also define, when $U$ varies among open subsets of $D$, a locally constant sheaf of pointed sets on~$D^\smooth$.
\end{proof}

\begin{corollaire}\label{cor:genericuniqueness}
Let $U\subset D$ be a smooth connected open subset and let $(\cM,\iso_{\wh U})$ and $(\cM',\iso'_{\wh U})$ be two elements of $\Gamma(U,\cH_D(\cM^\good))$. Then $(\cM,\iso_{\wh U})\simeq(\cM',\iso'_{\wh U})$ if and only if, for some germ $\gamma:(\CC,0)\to (X,D)$ transverse to $D$ at a point $\gamma(0)\in U$, we have $\gamma^+(\cM,\iso_{\wh U})\simeq\gamma^+(\cM',\iso'_{\wh U})$.
\end{corollaire}

\begin{proof}
According to Proposition \ref{prop:locconstgeneric}, this follows from the property that two sections on the connected set $U$ of a locally constant sheaf coincide if and only if they coincide at one point.
\end{proof}

\begin{corollaire}\label{cor:decomp}
Let $U\subset D$ be a smooth connected open subset and $(\cM,\iso_{\wh U})\in\Gamma(U,\cH_D(\cM^\good))$. The $D$-formal isomorphism $\iso_{\wh U}$ can be lifted (in a unique way) as an isomorphism $\iso_{U}:\cM\isom\cM^\good_{|U}$ if and only if for some germ $\gamma:(\CC,0)\to (X,D)$ transverse to $D$ at a point $\gamma(0)\in U$, $\gamma^*\iso_{\wh U}$ lifts as an isomorphism $\gamma^+\cM\isom\gamma^+\cM^\good$.
\end{corollaire}

\begin{proof}
Apply Corollary \ref{cor:genericuniqueness} with $(\cM',\iso'_{\wh U})=(\cM^\good_{|U},\id)$.
\end{proof}

\begin{corollaire}\label{cor:extensiontop}
Let $V\subset U$ be two connected nested open subsets contained in the smooth part of $D$. If the inclusion induces an isomorphism $\pi_1(V,x_o)\isom\pi_1(U,x_o)$ for some $x_o\in V$, then any $(\cM,\iso_{\wh D})_{V}\in\Gamma(V,\cH_D(\cM^\good))$ extends in a unique way as $(\cM,\iso_{\wh D})_{U}\in\Gamma(U,\cH_D(\cM^\good))$.
\end{corollaire}

\begin{proof}
The data $(\cM,\iso_{\wh D})_{V}$ correspond to a section $\sigma$ on $V$ of the locally constant sheaf of sets $\cH_D(\cM^\good)_{|D^\smooth}$. Under the assumption in the corollary, such a section extends in a unique way as a section on $U$.
\end{proof}

\subsection{Local constancy in arbitrary codimension}

Our aim is to prove the analogue of Proposition \ref{prop:locconstgeneric} along higher-codimensional strata of $D$. Let $x_o\in D$ with stratum $D^{(x_o)}$ of codimension $\ell\geq2$.

\begin{proposition}\label{prop:locconstcodimtwo}
The sheaf $\cH_D(\cM^\good)_{|D^{(x_o)}}$ is a locally constant sheaf of pointed sets, and its fiber at a smooth point $x_o$ is in bijection with $\cH_{\gamma^{-1}(D)}(\gamma^+\cM^\good)$ for any germ $\gamma:(\CC^\ell,0)\to (X,x_o)$ transverse to $D^{(x_o)}$.
\end{proposition}

\begin{proof}
We will use the same strategy as in the second proof of Proposition \ref{prop:locconstgeneric}. However, we will use an induction with respect to the rank of $\cM^\good$ in order to take care of the level structure.

Let us consider the sheaf $\cH_{D^{(x_o)}}(\cM^\good_{|D^{(x_o)}})$ whose sections on an open set $U\subset D^{(x_o)}$ consist of pairs of a germ of a $D$-meromorphic flat bundle $\cM$ on $U$ and an isomorphism $\iso_{\wh{U}}:\cM_{\wh{U}}\isom\cM^\good_{\wh{U}}$ (the formalization is along $D^{(x_o)}$, not along $D$; that it forms a sheaf and not only a presheaf follows from the uniqueness of isomorphisms compatible with $\iso$). By the same argument and references as in the second proof of Proposition \ref{prop:locconstgeneric}, we find that this is a locally constant sheaf of pointed sets, whose fiber at $x_o$ is identified by $\gamma^+$ with $\cH_0(\gamma^+\cM^\good_{|0})$. We are thus led to proving the following:
\begin{enumerate}
\item
If $U\subset D^{(x_o)}$ is simply connected and $x\in U$, then $\cM_{\wh{U}}$ has $\cM^\good_{|U}$ as a \emph{$D$\nobreakdash-formal} model if and only if a similar assertion holds for $\gamma^+\cM_{\wh{U}}$, where $\gamma$ is a transversal slice of $U$ at $x$.
\item
There is a one-to-one correspondence between the set of $\iso_{\wh D|U}$'s lifting $\iso_{\wh{U}}$ and the corresponding pullbacks by $\gamma^*$.
\end{enumerate}

The second point is clear since the sets consist of one element. For the first point, we then do not care about controlling the formal isomorphisms and we argue by induction on the rank of $\cM$. If $D(\Phi_{x_o})\neq D_{x_o}$, then the result follows from Example \ref{exem:decomp}. Assume now that $D(\Phi_{x_o})= D_{x_o}$. Then $\cM_{\wh{U}}$ has $\cM^\good_{|U}$ as a \emph{$D$-formal} model if and only if, for every $c\in C(x_o,\varphi_o)$, $\cM_{c|U}$ has $\cM^\good_{c|U}$ as a \emph{$D$-formal} model (\cf Proposition \ref{prop:devissage}). Since we can assume $\# C(x_o,\varphi_o)\geq2$ (\cf Lemma \ref{lem:Cgeq2}), the rank of each $\cM_c$ is strictly smaller than that of~$\cM$ and by induction the first point holds for every $\cM_{c|U}$. On the other hand, we also have that $\gamma^+\cM_{\wh{U}}$ has $\gamma^+\cM^\good_{|U}$ as a \emph{$D$-formal} model if and only if, for every $c\in C(x_o,\varphi_o)$, $\gamma^+\cM_{c|U}$ has $\gamma^+\cM^\good_{c|U}$ as a \emph{$\gamma^{-1}(D)$-formal} model. This gives the first point for $\cM_{|U}$.
\end{proof}

\subsection{Uniqueness results}\label{subsec:uniqueness}
We fix $\cM^\good$ as in \eqref{eq:gooddecbdle}, and we use the notation and definitions of Section \ref{subsec:settingnotation}.

\subsubsection*{Local uniqueness results}
We work in a neighbourhood of $x_o\in D$, so that $X$ is the product $B^m$ with coordinates $x_1,\dots,x_m$, where $B$ is a small disc in $\CC$, and $D$ is the divisor defined by $x_1\cdots x_\ell=0$ with components $D_i=\{x_i=0\}$. We regard $\Phi_{x_o}$ as a finite subset of $\CC\{x_1,\dots,x_n\}[(x_1\cdots x_\ell)^{-1}]/\CC\{x_1,\dots,x_n\}$, and we assume that it is good. Let $\varpi:\wt X\to X$ the real oriented blowing up of $D_1,\dots,D_\ell$. Let us also set $D_1^\circ:=\nobreak D_1\cap \nobreak D^\smooth$. In this setting we have $D^{(x_o)}=\{0\}^\ell\times B^{m-\ell}$ and
\begin{align*}
\wt X&\simeq(S^1)^\ell\times[0,\epsilon)^\ell\times B^{m-\ell},&\varpi^{-1}(D^{(x_o)})&\simeq(S^1)^\ell\times\{0\}^\ell\times B^{m-\ell},\\
\partial\wt X&\simeq(S^1)^\ell\times\partial[0,\epsilon)^\ell\times B^{m-\ell},&
\varpi^{-1}(D_1^\circ)&\simeq(S^1)^\ell\times\{0\}\times(0,\epsilon)^{\ell-1}\times B^{m-\ell}.
\end{align*}

The following result is due to J.-B.\,Teyssier who proved a much stronger statement, in the sense that it takes into account more structure on the sheaf of Stokes torsors and, moreover, it compares two $D$-meromorphic flat bundles that have the same good $\cO_{\wh{D^{(x_o)}}}$-model $\cM^\good$ and that are $\cO_{\wh D}$-isomorphic. The present statement will be enough for our purpose, and we will give an independent proof for the sake of completeness.

\begin{proposition}[J.-B.\,Teyssier, {\cite[Th.\,3 \& (2.3.1)]{Teyssier17}}]
\label{prop:uniquenesslocal}
Two germs $(\cM,\iso_{\wh D}),(\cM'\!,\iso'_{\wh D})$ at $x_o$ are isomorphic if and only if their restrictions to a curve transversal to $D^\smooth$ at one point~$x$ near $x_o$ are so.
\end{proposition}

As in the second proof of Proposition \ref{prop:locconstgeneric}, we will use the notion of Stokes-filtered local system. Recall that, in this setting, the restriction functor starting from non-ramified Stokes-filtered local systems $(\cL,\cL_\bbullet)$ on $\partial\wt X$ indexed by $\Phi_{x_o}$ (\cf Conven\-tion~\ref{conv:Phi}) to Stokes-filtered local systems on $\varpi^{-1}(D^{(x_o)})$ indexed by $\Phi_{x_o}$ is an equivalence (\cf \cite[Lem.\,3.17]{Mochizuki10b}, see also \cite[\S2.e]{Bibi16b}). A quasi-inverse functor is obtained by \begin{enumeratea}
\item\label{enum:uniquenesslocala}
taking the pullback local system $\cL$ by the natural projection forgetting the component $\partial[0,\epsilon)^\ell$,
\item\label{enum:uniquenesslocalb}
taking, for every $\cL_{\leq\varphi}\subset\cL$ ($\varphi\in\Phi_{x_o}$), its pullback in the pullback of $\cL$,
\item\label{enum:uniquenesslocalc}
and at each point of $\varpi^{-1}(D)$, summing in $\cL$ the various $\cL_{\leq\varphi}$'s for the $\varphi$'s that coincide near this point.
\end{enumeratea}

\begin{proof}[Proof of Proposition \ref{prop:uniquenesslocal}]
We fix $x\in D^\smooth$ and denote by $D_1^\circ$ the stratum of $x$ and by $\gamma:(\CC,0)\to(X,x)$ a curve transversal to $D_1^\circ$ at $x$.

We first note that it is enough to prove the following statement.

\begin{assertion}\label{ass:uniqueness}
A germ $\cM$ at $x_o$ that satisfies $\cM_{\wh D,x_o}\simeq\cM^\good_{\wh D,x_o}$ is uniquely determined by its restriction $\cM_{\wh D|D_1^\circ}$.
\end{assertion}

Indeed, if Assertion \ref{ass:uniqueness} is proved, assume $\gamma^+(\cM,\iso_{\wh D})\simeq\gamma^+(\cM',\iso'_{\wh D})$. By Corollary \ref{cor:genericuniqueness}, we obtain that $(\cM,\iso_{\wh D})_{|D_1^\circ}\simeq(\cM',\iso'_{\wh D})_{|D_1^\circ}$. According to the assertion, we then have $\cM\simeq\cM'$. Let us check that in such a case, $\iso_{\wh D}$ is also uniquely determined by its restriction to $D_1^\circ$. Since it can be represented by a matrix with entries in~$\cO_{\wh D}$, it is enough to prove the assertion for sections $f$ of $\cO_{\wh D}$ on a neighbourhood of~$x_o$, which follows then from the description of $\cO_{\wh D}$ recalled in the proof of Proposition \ref{prop:Malgrangeextension}. Indeed, setting $f=(f_i)_{i=1,\dots,\ell}$, the condition $f_{1|D_1^\circ}=0$ implies $f_1=0$, hence $\wh f_i=0$ in $\cO_{\wh x_o}$ for each~$i$, and therefore $f_i=0$ in $\cO_{\wh D_i,x_o}$.

We will thus prove Assertion \ref{ass:uniqueness}. Note that, if $D(\Phi_{x_o})\!\neq\!D_{x_o}$, Example \ref{exem:decomp} implies that $\cM\simeq\cM^\good$ and there is nothing to prove. However, note that the existence of some $\iso_{\wh D}$ is essential in this case. We will then assume that \hbox{$D(\Phi_{x_o})\!=\!D_{x_o}$}. We start with a particular case.

\begin{proof}[Proof of the assertion in a simple case]
We assume here that for all $\varphi\neq\psi\in\Phi_{x_o}$, $\varphi-\psi$ has poles along \emph{each} component of $D_{x_o}$. Let $(\cL,\cL_\bbullet)_{\varpi^{-1}(D_1^\circ)}$ be~the Stokes-filtered local system on $\varpi^{-1}(D_1^\circ)$ corresponding to $\cM_{|D_1^\circ}$. By our assumption on $\Phi_{x_o}$, it is the pullback by the natural projection
\[
\varpi^{-1}(D_1^\circ)=(S^1)^\ell\times\{0\}\times(0,\epsilon)^{\ell-1}\times B^{m-\ell}\to(S^1)^\ell\times\{0\}^\ell\times B^{m-\ell}=\varpi^{-1}(D^{(x_o)})
\]
of a Stokes-filtered local system $(\cL,\cL_\bbullet)_{\varpi^{-1}(D^{(x_o)})}$ on $\varpi^{-1}(D^{(x_o)})$, because in each summation procedure recalled in \eqref{enum:uniquenesslocalc} above, there is only one term. As a consequence, $(\cL,\cL_\bbullet)_{\varpi^{-1}(D^{(x_o)})}$ is uniquely determined from $(\cL,\cL_\bbullet)_{\varpi^{-1}(D_1^\circ)}$. We conclude that~$\cM$ is uniquely determined by its restriction to~$D_1^\circ$.
\end{proof}

Before giving the proof in general, let us recall the level structure of Stokes-filtered local systems indexed by $\Phi_{x_o}$ on $\partial\wt X$ (\cf \cite[\S\S2.6\,\&\,3.3]{Mochizuki08}, \cite[pp.\,41\,\&\,139]{Bibi10}). It is parallel to that for $D$-meromorphic flat bundles considered in Proposition \ref{prop:devissage}. In order to simplify the notation, we will assume that $0\in\Phi_{x_o}$, which we take as base point $\varphi_o$, and we use notation as in \eqref{eq:CPhi}. Let us start with a Stokes-filtered local system $(\cL,\cL_\bbullet)$ indexed by $\Phi_{x_o}$ on $\varpi^{-1}(D^{(x_o)})$. It induces a filtration $\cL_\preceq$ indexed by $C=C(\varphi_o)\subset\CC$, where the order $c\preceq c'$ on the latter set at a point of $\varpi^{-1}(D^{(x_o)})$ is defined as the property that $\exp((c-c')x^{-\bmm_o})$ has moderate growth near this point. By definition, we have
\[
\cL_{\preceq c}=\sum_{\substack{\psi\in\Phi_{x_o}\\ c(\psi)\leq c}}\cL_{\leq\psi}.
\]
Moreover, for every $c\in C$ (recall that we can assume $\#C\geq2$, \cf Lemma \ref{lem:Cgeq2}), the graded sheaf $\gr_c\cL$ is a locally constant sheaf, on which the filtration $\cL_\bbullet$ induces a Stokes-filtration whose jumps are contained in~$\Phi_{x_o}(c)$. Lastly, for $\psi\in\Phi_{x_o}(c)$, $\cL_{\leq\psi}$ is the pullback of $(\gr_c\cL)_{\leq\psi}$ by the \hbox{projection} $\cL_{\preceq c}\to\nobreak\gr_c\cL$. By the Riemann-Hilbert correspondence, $(\gr_c\cL,(\gr_c\cL)_\bbullet)$ corresponds to $\cM_c$ considered in Proposition \ref{prop:devissage}.

A similar structure is obtained for a Stokes-filtered local system on $\varpi^{-1}(D_1^\circ)$, by taking pullback and summation of $(\cL,\cL_\bbullet)_{\varpi^{-1}(D_1^\circ)}$ as above.

\subsubsection*{Proof of Assertion \ref{ass:uniqueness} in the general case}
We argue by induction on the rank $r$ of~$\cM$. It remains to consider the case where $D(\Phi_{x_o})=D_{x_o}$. Given $\cM,\cM'$ with $\cM_{\wh D}\simeq\cM'_{\wh D}\simeq\cM^\good_{\wh D}$, assume that $\cM_{|D_1^\circ}\simeq\cM'_{|D_1^\circ}$. It follows that the Stokes-filtered local systems $(\cL,\cL_\bbullet)_{\varpi^{-1}(D_1^\circ)},(\cL',\cL'_\bbullet)_{\varpi^{-1}(D_1^\circ)}$ are isomorphic, hence we have an isomorphism on $\varpi^{-1}(D_1^\circ)$:
\[
\bigl(\cL,\cL_\preceq,(\gr_c\cL,(\gr_c\cL)_\bbullet)_{c\in C}\bigr)_{\varpi^{-1}(D_1^\circ)}\overset{\textstyle(*)}\simeq\bigl(\cL',\cL'_\preceq,(\gr_c\cL',(\gr_c\cL')_\bbullet)_{c\in C}\bigr)_{\varpi^{-1}(D_1^\circ)}.
\]
Arguing as in the special case for $\cL_\preceq$ instead of $\cL_\leq$, we obtain an isomorphism $(\cL,\cL_\preceq)_{\varpi^{-1}(D^{(x_o)})}\simeq(\cL',\cL'_\preceq)_{\varpi^{-1}(D^{(x_o)})}$ and, for each $c\in C$, the induced isomorphism $\gr_c\cL_{\varpi^{-1}(D^{(x_o)})}\simeq\gr_c\cL'_{\varpi^{-1}(D^{(x_o)})}$ corresponds to that induced by~$(*)$ (by pullback by the projection). Since $\cM_{c,\wh D}\simeq\cM'_{c,\wh D}\simeq\cM^\good_{c,\wh D}$ have rank $<r$, we conclude by induction that the component
\[
(*)_c:(\gr_c\cL,(\gr_c\cL)_\bbullet)_{\varpi^{-1}(D_1^\circ)}\isom(\gr_c\cL',(\gr_c\cL')_\bbullet)_{\varpi^{-1}(D_1^\circ)}
\]
of $(*)$ comes from an isomorphism $\cM_c\simeq\cM'_c$, that is, an isomorphism on $\varpi^{-1}(D^{(x_o)})$:
\[
(\gr_c\cL,(\gr_c\cL)_\bbullet)_{\varpi^{-1}(D^{(x_o)})}\simeq(\gr_c\cL',(\gr_c\cL')_\bbullet)_{\varpi^{-1}(D^{(x_o)})}.
\]
As a consequence, the isomorphism $\cL_{\preceq c}\isom\cL'_{\preceq c}$ sends $\cL_{\leq\psi}$ isomorphically to $\cL'_{\leq\psi}$ on $\varpi^{-1}(D^{(x_o)})$ for any $\psi\in\Phi_{x_o}(c)$: indeed, we have a commutative diagram
\[
\xymatrix@R=2mm{
\cL_{\preceq c}\ar[r]^-\sim\ar[dd]&\cL'_{\preceq c}\ar[dd]\\&\\
\gr_c\cL\ar[r]^-\sim&\gr_c\cL'\\
(\gr_c\cL)_{\leq\psi}\ar[r]^-\sim\ar@{}[u]|\cup&(\gr_c\cL')_{\leq\psi}\ar@{}[u]|\cup
}
\]
and $\cL_{\leq\psi}$ (\resp $\cL'_{\leq\psi}$) is the pullback of $(\gr_c\cL)_{\leq\psi}$ (\resp $(\gr_c\cL')_{\leq\psi}$) by the projection $\cL_{\preceq c}\to\gr_c\cL$ (\resp $\cL'_{\preceq c}\to\gr_c\cL'$).

We conclude that $(\cL,\cL_\bbullet)\simeq(\cL',\cL'_\bbullet)$ on $\varpi^{-1}(D^{(x_o)})$, as was to be proved.
\end{proof}

\subsubsection*{Global results}
We now go back to the global setting of Section \ref{subsec:settingnotation}. In particular, $D$~is connected and $\cM^\good$ is fixed. The following result is due to J.-B.\,Teyssier, who has obtained a stronger form.

\begin{corollaire}[J.-B.\,Teyssier {\cite[Proof of Th.\,4]{Teyssier17}}]\label{cor:uniquenessglobal}
Given $\cM^\good$-marked $D$\nobreakdash-mero\-mor\-phic flat bundles $(\cM,\iso_{\wh D})$ and $(\cM',\iso'_{\wh D})$ on $X$, we have $(\cM,\iso_{\wh D})\simeq(\cM',\iso'_{\wh D})$ if and only if there exists a point $x$ in $D^\smooth$ and a germ of curve $\gamma:(\CC,0)\to(X,x)$ transverse to $D$ at $x$ such that $(\gamma^+\cM,\iso_{\wh0})\simeq(\gamma^+\cM',\iso'_{\wh0})$.
\end{corollaire}

\begin{proof}
For any two germs $(\cM,\iso_{\wh D}),(\cM',\iso'_{\wh D})$, let $D'\subset D$ be the subset of points where they are isomorphic. Assume that it is not empty. It is open by definition. Let us show that it is closed. Let $x_o\in\ov{D'}$.
\begin{itemize}
\item
Either $x_o\in D^\smooth$, in which case $D'$ contains a point $x'$ in the component of $D^\smooth$ containing~$x_o$, hence also contains this whole component according to Corollary \ref{cor:genericuniqueness}, and therefore $D'\ni x_o$,
\item
or $x_o\notin D^\smooth$, in which case we apply Proposition \ref{prop:uniquenesslocal} to also conclude that $x_o\in D'$.
\end{itemize}
Therefore, $D'$ is closed, hence equal to $D$. The ``only if'' part of the proposition is clear, and for the ``if'' part we apply the previous argument, since we know by Corollary \ref{cor:genericuniqueness} that $D'\neq\emptyset$.
\end{proof}

We will make use of the following consequence, which is a weaker version of \cite[Th.\,3]{Teyssier17}. Although its statement is local, it relies on the global result of Corollary \ref{cor:uniquenessglobal}.

\begin{corollaire}[J.-B.\,Teyssier]\label{cor:uniquenessglobalI}
Assume that $D$ is smooth. Let $\Psi_{x_o}\subset\cO_{X,x_o}(*D)/\cO_{X,x_o}$ be a (not necessarily good) finite subset, and, for any $\psi\in\Psi_{x_o}$, let $\cR_\psi$ be a germ of regular $D$-meromorphic flat bundle at $x_o$. Set $\cN=\bigoplus_{\psi\in\Psi_{x_o}}(\cE^\psi\otimes\cR_\psi)$. Let $(\cM,\iso_{\wh D}),(\cM',\iso'_{\wh D})$ be two $\cN$\nobreakdash-marked $D$-meromorphic flat bundles in the neighbourhood of $x_o$ and let \hbox{$\gamma:(\CC,0)\to(X,D)$} be a germ of curve whose image is not contained in $D$. Then
\[
(\cM,\iso_{\wh D})\simeq(\cM',\iso'_{\wh D})\iff\gamma^+(\cM,\iso_{\wh D})\simeq\gamma^+(\cM',\iso'_{\wh D}).
\]
\end{corollaire}

\begin{proof}
The point is to prove the implication $\Leftarrow$ and we assume that the isomorphism of the right-hand side holds for some curve $\gamma$. We shall denote by $X$ a small (connected) neighbourhood of $x_o$ where all data are defined. Let us first assume that $\gamma(0)=x_o$. There exists (\cf \cite[Lem.\,9.11]{Bibi10}) a projective modification $\modif:X'\to X$ that is an isomorphism away from~$D$ and the normal crossing divisor $D':=\modif^{-1}(D)$ such that $\modif^+\cN$ is good. This result is much easier than the general result on the resolution of turning points obtained by K.\,Kedlaya \cite{Kedlaya10} (and T.\,Mochizuki \cite{Mochizuki08} in the algebraic case). We can moreover assume that~$\gamma$ lifts as a curve~$\gamma'$ that is transverse to $\modif^{-1}(x_o)$ at a smooth point. We can apply Corollary \ref{cor:uniquenessglobal} to $\modif^+(\cM,\iso_{\wh D}),\modif^+(\cM',\iso'_{\wh D})$, and deduce that both are isomorphic. It follows from Lemma \ref{lem:modifHD} that $(\cM,\iso_{\wh D}),(\cM',\iso'_{\wh D})$ are isomorphic.

Assume now that $x:=\gamma(0)\neq x_o$. The first part of the proof applied at $x$ shows that $(\cM,\iso_{\wh D})_{|\nb(x)}\simeq(\cM',\iso'_{\wh D})_{|\nb(x)}$. Let $U$ be the open subset of $D$ on which $\cN$ is good. This is the complement in $D$ of the union of the zero sets of the meromorphic functions $\psi-\eta$ for $\psi,\eta\in\Psi_{x_o}$ and $\psi\neq\eta$. Since $D$ is smooth, this set is connected. By Corollary \ref{cor:genericuniqueness}, we deduce an isomorphism $(\cM,\iso_{\wh D})_{|U}\simeq(\cM',\iso'_{\wh D})_{|U}$. We conclude with Lemma \ref{lem:Hartogs}.
\end{proof}

\section{Proof of the main results}\label{sec:proofs}
We consider the setting and notation of Theorem~\ref{th:extension}.

\subsection{The most degenerate case}
If all coordinates of $t_o$ coincide, there is only one class $(\gamma_{t_o}^+\cM,\gamma_{t_o}^+\iso_{\wh T})$ since, up to a twist by $\cE^{t_{o,1}/z}$, $\gamma_{t_o}^+\cN$ has a regular singularity at $z=0$. The restriction induces then a surjective map between the two corresponding sets of isomorphism classes, and the injectivity is a consequence of the following more precise proposition.

\begin{proposition}\label{prop:mostdegenerate}
Assume that the connected open subset $U$ contains such a $t_o$. Let $(\cM_U,\iso_{\wh U})\in\Gamma(U,\cH_T(\cN))$. Then $\iso_{\wh U}$ can be lifted (in a unique way) as an isomorphism $\iso_U:\cM_U\isom\cN_{|U}$.
\end{proposition}

\begin{proof}
We will argue when $t_o$ is the origin of $T$, the other cases being obtained by an exponential twist. Let $\modif:X'\to X$ be the blowing-up of the origin in $X$ and let $T'$ be the strict transform of $T$ by $\modif$, so that $\modif_{|T'}:T'\to T$ is nothing but the blowing-up of~$T$ at the origin. Let $E=\modif^{-1}(0)\simeq\PP^n$ be the exceptional divisor, so that $D:=E\cup T'$ is a divisor with normal crossings.

There is a chart with coordinates $(u_1,\dots,u_n,\zeta)$ such that $E$ is defined as $\zeta=\nobreak0$ and~$\modif$ is given by $(u_1,\dots,u_n,\zeta)\mto(u_1\zeta,\dots,u_n\zeta,\zeta)$, so that $\modif^+\cN$ has regular singularities along $E$ away from $E\cap T'\simeq\PP^{n-1}$. On the other hand, one checks that $\modif^+\cN$ is good on a Zariski dense open subset of $E\cap T'$, \eg by computing the meromorphic functions $\modif^*((t_i-t_j)/z)$ in a chart with coordinates $(v_1,\dots,v_n,\zeta)$ where $\modif$ is given (say) by $(v_1,\dots,v_n,\zeta)\mto(v_1,v_1v_2,\dots,v_1v_n,v_1\zeta)$.

Let $t'_o$ be a point in this open subset of $E\cap T'$, which we can also consider as a point in $X'$. At this point, $\modif^+\cN$ has thus regular singularities along one of the components of $D$ passing through~$t'_o$. From Example \ref{exem:decomp} we conclude that $\modif^+\cM_U\simeq\modif^+\cN$ in some open neighbourhood of $t'_o$. Intersecting this neighbourhood with $U\moins\Delta$ gives a non-empty open subset in $U\moins\Delta$ where $\iso_{\wh U}$ can be lifted as an isomorphism $\cM_U\isom\cN$. We deduce from Corollary \ref{cor:decomp}, applied to the connected open subset $U\moins\Delta$, that $\iso_{\wh U}$ can be lifted as an isomorphism $\iso_{U\moins\Delta}:\cM_{|U\moins\Delta}\isom\cN_{|U\moins\Delta}$. We now conclude with Lemma \ref{lem:Hartogs}.
\end{proof}

\subsection{Proof of a variant of Theorem \ref{th:extension} in a special case}\label{subsec:proofvariant}
As an example for the method of proof of Theorem~\ref{th:extension}, we develop in this section a low-dimensional case, obtained by a generic two-dimensional slice of the pair $(T,\Delta)$ considered in the general case. 

The setting is as follows. We have $\dim T=2$, with coordinates $t=(t_1,t_2)$. We fix $a,b\in\CC$ such that $a,b,a-b\neq0$ and we consider the functions
\[
f_1(t)=t_1,\quad f_2(t)=t_1+t_2,\quad f_3(t)=t_1+a,\quad f_4(t)=t_2+b,
\]
and the elementary $T$-meromorphic flat bundle
\[
\cN=\bigoplus_{i=1}^4(\cE^{-f_i(t)/z}\otimes\cR_i),
\]
with $\cR_i$ regular along $T=\{z=0\}$. Then $\cN$ is good away from $\Delta=\{t_2=0\}$, so that the stratum is reduced to $\Delta$ and we set $U=T$. We also set
\[
\cN_o=\bigl(\cE^{-f_1(t)/z}\otimes(\cR_1\oplus\cR_2)\bigr)\oplus\bigoplus_{i=3}^4(\cE^{-f_i(t)/z}\otimes\cR_i),
\]
so that $\cN_o$ is a good decomposed $T$-meromorphic flat bundle.

We denote by $\modifu:X_1\to X$ the blowing-up of the origin in $X$. We set \hbox{$E_1\!=\!\modifu^{-1}(0)\!\simeq\!\PP^2$}, we denote by $T_1$ the strict transform of $T$ by $\modifu$, so that $\modif_{1|T_1}:T_1\to T$ is the blowing-up of the origin in $T$. The strict transform $\Delta_1$ of $\Delta$ is contained in $T_1$ and intersects $\modif_{1|T_1}^{-1}(0)=E_1\cap T_1\simeq\PP^1$ at one point $\delta_1$. The restriction $\modif_{1|\Delta_1}:\Delta_1\to \Delta$ is an isomorphism.

\begin{figure}[htb]
\begin{center}
\includegraphics[width=\textwidth]{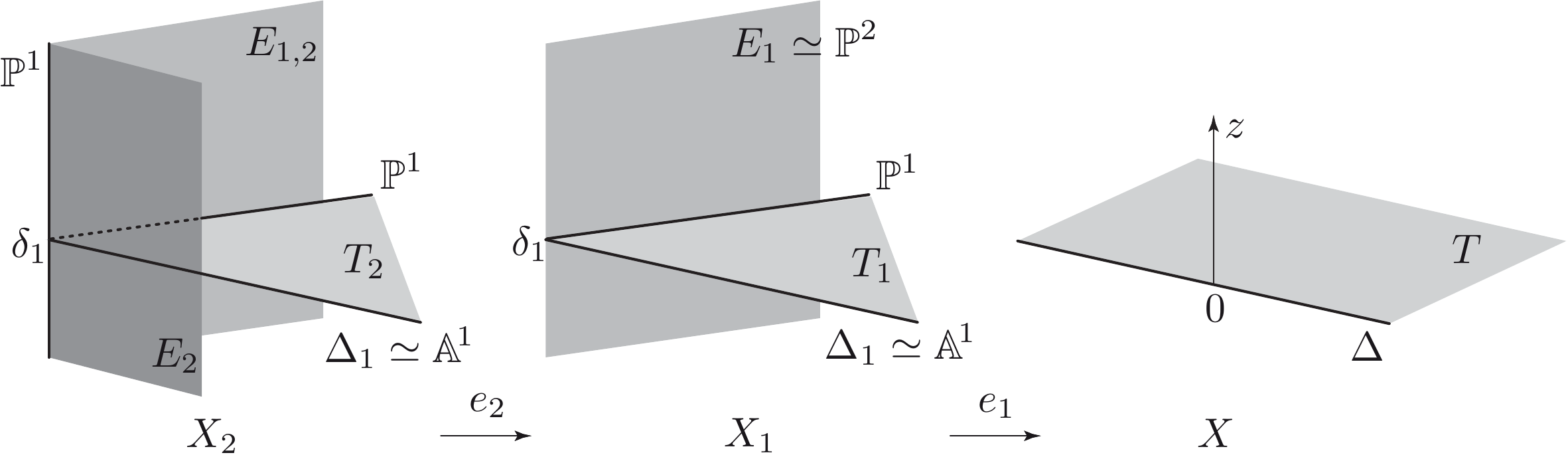}
\caption{\label{fig:X2}Simplified representation of the blowing-ups $e_1$ and $e_2$.}\vspace*{-.7\baselineskip}
\end{center}
\end{figure}

We denote by $\modifd:X_2\to X_1$ the blowing-up of $\Delta_1$ in $X_1$. We set
\[
E_2=\modifd^{-1}(\Delta_1)\simeq\PP^1\times\Delta_1,
\]
we denote by $E_{1,2}$ the strict transform of $E_1$, so that $\modif_{2|E_{1,2}}:E_{1,2}\to E_1$ is the blowing-up of $\delta_1$ in $E_1$ and $\modif_{2|E_{1,2}}^{-1}(\delta_1)\simeq\PP^1\times\{\delta_1\}$. It is standard to check that $E_{1,2}$, being the complex blow-up of $\PP^2$ at one point, is simply connected.\footnote{One can regard $E_{1,2}$ as the union of $\CC^2$ and the union of two $\PP^1$ that intersect at one point; both sets are simply connected, and the second one has an open neighbourhood $W$ that retracts onto it, hence is also simply connected. Since $W\cap\CC^2$ is connected, the van Kampen theorem gives the result.}  Let $T_2$ be the strict transform of $T_1$. Then $\modif_{2|T_2}:T_2\to T_1$ is an isomorphism since it is the blowing-up of~$\Delta_1$ of codimension one in~$T_1$. We also have $T_2\cap E_{1,2}\simeq\PP^1$ since it is the blow-up of $\PP^1$ at $\delta_1$. We thus regard~$\delta_1$ as a point in $X_2$ and $\Delta_1$ as the subset $T_2\cap E_2$ in~$X_2$. The geometric setting at $\delta_1$ in~$X_2$ is pictured in Figure \ref{fig:X2}.

We set $\modif=\modifu\circ\modifd:X_2\to X$. We denote by $\varpi:\wt X_2\to X_2$ the real oriented blowing-up of~$X_2$ along the components $(T_2,E_{1,2},E_2)$ of $D:=\modif^{-1}(T)$.

\begin{lemme}\label{lem:StDNo}
The sheaf $\St_D(\modif^+\cN_o)$ is locally constant on $\modif^{-1}(T)$.
\end{lemme}

\begin{proof}
Using the notation of Section \ref{subsec:verygoodformal}, we note that the pushforward $\modif_*$ induces an isomorphism $\cH_D(\modif^+\cN_o)\isom\modif^{-1}\cH_T(\cN_o)$. By Proposition \ref{prop:locconstgeneric}, $\cH_T(\cN_o)$ is a locally constant sheaf on $T$, hence so is $\cH_D(\modif^+\cN_o)$ on $D$. The assertion follows then from Theorem \ref{th:HSt}.
\end{proof}

\begin{lemme}\label{lem:StDN}
The $D$-meromorphic flat bundle $\modif^+\cN$ is good and the sheaf $\St_D(\modif^+\cN)$ is constant when restricted to $E_{1,2}$.
\end{lemme}

\begin{proof}
Since $E_{1,2}$ is simply connected, it is enough to prove local constancy. Let us write down the charts of the various blow-ups. We cover $X_1$ by three charts $X_1(i)$ ($i=1,2,3$), with coordinates $(u_i,v_i,\zeta_i)$ so that $\modif_1$ is given respectively by the formulas:
\[
X_1(1):
\begin{cases}
t_1=u_1\zeta_1,\\
t_2=v_1\zeta_1,\\
z=\zeta_1,\\
\end{cases}
\quad
X_1(2):
\begin{cases}
t_1=u_2v_2,\\
t_2=v_2,\\
z=v_2\zeta_2,\\
\end{cases}
\quad
X_1(3):
\begin{cases}
t_1=u_3,\\
t_2=u_3v_3,\\
z=u_3\zeta_3.\\
\end{cases}
\]
We note that the charts $X_1(1)$ and $X_1(2)$ can be regarded as contained in $X_2$, since they do not intersect the center $\Delta_1=\{u_3=\zeta_3=0\}$ of the blowing-up $\modif_2$.
\begin{enumerate}
\item\label{pf:StDN1}
In the chart $X_1(1)$, we have $X_1(1)\cap T_1=\emptyset$ and $X_1(1)\cap E_1=\{\zeta_1=0\}$. One checks that $\modif_1^+\cN$ is good there, and more precisely $\exp\modif_1^*(-t_2/z)$ is holomorphic there and the block-diagonal morphism $(\id,\exp\modif_1^*(-t_2/z)\id,\id,\id)$ induces an isomorphism between $\modif_1^+\cN_o$ and $\modif_1^+\cN$.
\item\label{pf:StDN2}
In the chart $X_1(2)$, we have $X_1(2)\cap T_1=\{\zeta_2=0\}$ and \hbox{$X_1(2)\cap E_1=\{v_2=0\}$}. One checks that $\modif_1^+\cN$ is good there. Although $\exp\modif_1^*(-t_2/z)=\exp(-1/\zeta_2)$ is not holomorphic near $E_1\cap T_1$, we claim that the block-diagonal morphism corresponding to $(\id,\exp\modif_1^*(-t_2/z)\id,\id,\id)$ induces an isomorphism between $\cAut^\rdD(\modif_1^+\cN_o)$ and $\cAut^\rdD(\modif_1^+\cN)$ on $\varpi^{-1}(E_1\cap T_1)$. Indeed, this morphism only affects the blocks $ij$ with $i\neq j$ and $i$ or $j$ equal to $2$. The blocks $12$ and $21$ of a section of $\cAut^\rdD(\modif_1^+\cN)$ are zero, since $\exp\modif_1^*((f_2-f_1)/z)=\exp\modif_1^*(-t_2/z)$ is nowhere of rapid decay, and so are the blocks $12$ and $21$ of $\cAut^\rdD(\modif_1^+\cN_o)$. For the blocks $23$, $32$, $24$ and $34$, which may be nonzero, multiplying by $\exp\modif_1^*(-t_2/z)$ does not affect the leading term of the exponential, hence neither does it affect the rapid decay condition.
\end{enumerate}

At this step, we have proved that $\cAut^\rdD(\modif_1^+\cN_o)$ and $\cAut^\rdD(\modif_1^+\cN)$ are isomorphic on $\varpi^{-1}(E_{1,2})\moins\varpi^{-1}(E_2)$, hence so are $\St_D(\modif^+\cN_o)$ and $\St_D(\modif^+\cN)$ on $E_{1,2}\moins E_2$.

\begin{enumerate}\setcounter{enumi}{2}
\item\label{pf:StDN3}
We now blow up the chart $X_1(3)$ along the ideal $(v_3,\zeta_3)$ of $\Delta_1$, giving rise to the charts $X_2(3\textup a)$ and $X_2(3\textup b)$ of $X_2$, with respective coordinates $(u_3,w_1,\eta_1),(u_3,w_2,\eta_2)$ satisfying
\[
X_2(3\textup a):
\begin{cases}
v_3=w_1\eta_1,\\
\zeta_3=\eta_1,
\end{cases}
\quad
X_2(3\textup b):
\begin{cases}
v_3=w_2,\\
\zeta_3=w_2\eta_2.
\end{cases}
\]
\end{enumerate}
Then one checks that $\modif^+\cN$ is good and the same argument as given in \eqref{pf:StDN2} above gives an isomorphism on $\varpi^{-1}(E_{1,2})$ between $\cAut^\rdD(\modif^+\cN_o)$ and $\cAut^\rdD(\modif^+\cN)$, hence $\St_D(\modif^+\cN_o)$ and $\St_D(\modif^+\cN)$ are isomorphic on $E_{1,2}$. The lemma follows from Lemma \ref{lem:StDNo}.
\end{proof}

\subsubsection*{Proof of the surjectivity of $\gamma_{t_o}^+$}
Recall that, here, $S_o=\Delta$ and that we have fixed $t_o\in\Delta$.

\subsubsection*{Step $1$: extension to $E_{1,2}$}
The curve $\gamma_{t_o}$ lifts as a curve $\gamma_{t'_o}:(\CC,0)\to(X_2,t'_o)$ transverse to $E_{1,2}$ at a point $t'_o$ in $D^\smooth$. We have $\cH_0(i_{t_o}^+\cN)=\cH_0(i_{t'_o}^+\modif^+\cN)$. By Proposition \ref{prop:locconstgeneric}, this set is identified with $i_{t'_o}^{-1}\cH_D(\modif^+\cN)$. Since $\cH_D(\modif^+\cN)_{|E_{1,2}}$ is constant, according to Lemma \ref{lem:StDN} and Theorem \ref{th:HSt}, this set is also identified with $\Gamma(E_{1,2},\cH_D(\modif^+\cN))$. An element of $\cH_0(i_{t_o}^+\cN)$ defines thus a pair $(\cM,\iso_{\wh D})$ on some neighbourhood of $E_{1,2}$ in $X_2$.

\subsubsection*{Step $2$: extension to $E_{1,2}\cup\Delta_1$}
Now we apply Proposition \ref{prop:locconstcodimtwo} to our section of $\cH_D(\modif^+\cN)_{|E_{1,2}}$. Since $\Delta_1$ retracts to a neighbourhood of $\delta_1$, the section extends to a neighbourhood of $\Delta_1$ in a unique way.

\subsubsection*{Step $3$: extension to $E_{1,2}\cup E_2$}
The restriction to \hbox{$(E_{1,2}\cap E_2)\!\moins\!\{\delta_1\}$} of the section of $\cH_D(\modif^+\cN)_{|E_{1,2}}$ constructed in Step~$1$ extends uniquely as a section of $\cH_D(\modif^+\cN)_{E_2\moins\Delta_1}$ by Proposition \ref{prop:locconstgeneric}. The restriction of the section constructed in Step~$2$ to a punctured neighbourhood of $\Delta_1$ in~$E_2$ coin\-cides with the restriction to this punctured neighbourhood of this new extension, by uniqueness, since they coincide in the neighbourhood of $\delta_1$.

\subsubsection*{Step $4$: extension to $D=E_{1,2}\cup E_2\cup T_2$}
The section constructed in Step~$3$ exists in some neighbourhood of $E_{1,2}\cup E_2$ in $D$. There exists a fundamental basis of neighbourhoods $U$ of $E_{1,2}\cup E_2$ in $D$ such that the inclusion $U\subset D$ induces an isomorphism $\pi_1(U)\isom\pi_1(D)$, as seen by taking the pullback by $\modif$ of a suitable basis of neighbourhoods of~$\Delta$ in $T$. By arguing like in Corollary \ref{cor:extensiontop}, the section constructed in Step~$3$, defined on such a neighbourhood $U$, extends in a unique way to a section of $\cH_D(\modif^+\cN)$ defined on $D$. Lemma \ref{lem:modifHD} enables us to conclude.\qed

\subsection{Proof of Theorem \ref{th:extension}}
We now consider the general case in Theorem \ref{th:extension}. Our first aim is to prove, in the context of Theorem~\ref{th:extension}, an analogue of Propositions \ref{prop:locconstgeneric} and \ref{prop:locconstcodimtwo} on each stratum of $\Delta$. Let $t_o\in\Delta\subset T$ and let $S(t_o)$ be its stratum.

\begin{proposition}\label{prop:locconstdegenerate}
When restricted to $S(t_o)$, the sheaf $\cH_T(\cN)$ is a locally constant sheaf of pointed sets, and its fiber at $t_o$ is in bijection with $\cH_0(\gamma_{t_o}^+\cN)$.
\end{proposition}

\begin{proof}
We fix $t_o\in\Delta\subset T$ and we work locally at $t_o$. We can decompose $\{1,\dots,n\}$ as $\bigsqcup_{r\in R}I_r$ such that, for every $r\in R$, we have $\{i,j\}\subset I_r$ if and only if $t_{o,i}=t_{o,j}$. For each $r\in R$, we choose an element in $I_r$ that we denote by~$r$ and we set $I'_r=I_r\moins\{r\}$. We then~set $p=\#R=\dim S(t_o)$, $m=n-p$, and
\begin{equation}\label{eq:cNo}
\cN_o=\bigoplus_{r\in R}(\cE^{-t_r/z}\otimes\cR_{I_r}),\quad \cR_{I_r}:=\bigoplus_{i\in I_r}\cR_i.
\end{equation}
We fix a neighbourhood of $t_o$ in $T$ of the form $V\times W$, such that $V$ has the coordinates $(t_r)_{r\in R}$ and $W$ has coordinates $(\tau_i^r)_{r\in R,\,i\in I'_r}$, so that for $r\in R$ and $i\in I'_r$, we have $t_i=t_r+\tau_i^r$, and small enough such that $\Delta\cap (V\times W)$ is given by the equations $\tau_i^r=0$ ($r\in R,\,i\in I'_r$) and $\tau_i^r-\tau_j^r$ ($r\in R,\,i\neq j\in I'_r$). We have $S(t_o)\cap (V\times W)=V\times\{0\}$. We now denote this neighbourhood by $T$, and set $X=(V\times W)\times\CC_z=:V\times Y$.

Let $\modif_1:X_1\to X$ be the blowing-up of this stratum, \ie that of the ideal
\[
\bigl((\tau_i^r)_{r\in R,\,i\in I'_r},z\bigr).
\]
We have $X_1=V\times Y_1$ with obvious notation. Every object below is a product of~$V$ with the corresponding object in $Y_1$. The pullback $D_1:=\modif_1^{-1}(T)$ is the union of the exceptional divisor $E_1=\modif_1^{-1}(S(t_o))$ (in this local setting, $E_1\simeq S(t_o)\times\PP^m$), and the strict transform $T_1$ of $T$. The exceptional divisor of $\modif_{|T_1}:T_1\to T$ is equal to $E_1\cap T_1\simeq S(t_o)\times\PP^{m-1}$. Moreover, $T_1$ is a disc-bundle in the normal bundle of $E_1\cap T_1$ in $T_1$. Similarly, the strict transform $\Delta_1\subset T_1$ of $\Delta$ is a disc-bundle over $\delta_1:=E_1\cap\Delta_1$ (the product of an arrangement of projective hyperplanes in~$\PP^{m-1}$ with~$S(t_o)$). We~denote by $\Delta'_1$ the union of the two-by-two intersections of the components of $\Delta_1$ (it has codimension three in $X_1$) and by $\delta'_1$ its intersection with~$E_1$. Lastly, we set $X_1^\circ=X_1\moins\Delta'_1$, and similarly for the other objects. In particular, $\delta_1^\circ$ is non-singular and $\Delta_1^\circ$ is a disc-bundle over it.

We now argue as for the simple case of Section \ref{subsec:proofvariant}. We denote by $\modif_2:X_2^\circ\to X_1^\circ$ the blowing-up of $\Delta_1^\circ$ in $X_1^\circ$ and we obtain $D_2^\circ:=E_2^\circ\cup E_{1,2}^\circ\cup T_2^\circ$. Then Lemmas \ref{lem:StDNo} and \ref{lem:StDN} hold in this setting. The isomorphism to be considered is block diagonal with blocks indexed by $R$, each block having the form $(\id,(\exp(-\modif^*\tau_i^r)\id)_{i\in I'_r})$.

For the surjectivity of $\gamma_{t_o}^+$, there is no change to be done in Step~$1$ of Section \ref{subsec:proofvariant}. For the other steps, we use the disc-bundle structure over the intersection with $E_{1,2}^\circ$ of all the objects involved, instead of the structure of a product with $\CC$. The same homotopy argument applies.

Let us now conclude with the surjectivity of $\gamma_{t_o}^+$. Given $(\cM^{t_o},\iso_{\wh0})$ with model $\gamma_{t_o}^+\cN$, corresponding to an element of
\[
\cH_0(\gamma_{t_o}^+\cN)\simeq\cH_0(\gamma_{t'_o}^+\modif^+\cN)\simeq\cH_{D_2^\smooth}(\modif^+\cN)_{t'_o},
\]
we have extended it in a unique way as a section of $\cH_{D_2^\circ}(\modif^+\cN)$, which corresponds thus to a pair $(\cM_2^\circ,\iso_{\wh{D_2^\circ}})$ with model $\modif^+\cN$ satisfying $\gamma_{t'_o}^+(\cM_2^\circ,\iso_{\wh{D_2^\circ}})=(\cM^{t_o},\iso_{\wh0})$. Applying $\modif_{2+}$ and according to Lemma \ref{lem:modifHD}, we obtain a pair $(\cM_1^\circ,\iso_{\wh{D_1^\circ}})$ on $X_1^\circ$. Owing to the theorem of B.\,Malgrange (Proposition \ref{prop:Malgrangeextension}), this pair extends in a unique way as a pair $(\cM_1,\iso_{\wh{D_1}})$ on $X_1$. Lastly, applying $\modif_{1+}$ and according to Lemma \ref{lem:modifHD}, we obtain a pair $(\cM,\iso_{\wh T})$ as wanted.

On the other hand, injectivity of $\gamma_{t_o}^+$ is given by the proof above, since there is no choice in any extension procedure. This shows that, on $V$, the sheaf $\cH_T(\cN)_{|S(t_o)}$ is constant, with fibre given by applying $\gamma_{t_o}^+$.
\end{proof}

\begin{proof}[End of the proof of Theorem \ref{th:extension}]
We first define the notion of a star-shaped open set. Let $S_o$ be a stratum of $\Delta$ and set $S_o^\star=\bigcup_{\ov S\supset S_o}\ov S$ be its star (where $S$ varies in the set of strata of the natural stratification of $\Delta$, so that $\ov S$ is a linear subspace). By choosing the coordinates as in the proof of Proposition \ref{prop:locconstdegenerate}, we find a product decomposition $S_o^\star\!\simeq\!S_o\!\times\!\CC^m$, and we consider the corresponding projection $p_o\,{:}\,S_o^\star\!\to\!S_o$. We endow~$\CC^m$ with its standard Euclidean metric.

\begin{definition}\label{def:starshaped}
An open subset $U\subset T$ is said to be \emph{star shaped} with respect to $U\cap\nobreak S_o$ if $U$ contains the Euclidean ball centered at the origin of the linear subspace $\ov{S(t)}\cap\nobreak p_o^{-1}(p_o(t))$ containing $t$, for any $t\in U$.
\end{definition}

For a star-shaped open set $U$ with respect to $U\cap S_o$, the flow of the radial vector field in each stratum $S$ containing $S_o$ in its closure induces a deformation retraction of $U\cap S$ to $U\cap S_o$. On the other hand, there exists a fundamental system of open neighbourhoods $V_S$ of $U\cap S_o$ in $U\cap \ov S$ that are star shaped. So $V_S\cap S\subset U\cap S$ induces an isomorphism of fundamental groups.

Assume we are given $(\cM^{t_o},\iso_{\wh0})$ with model $\gamma_{t_o}^+\cN$ for $t_o\in S_o$. Since $\cH_T(\cN)_{|S_o}$ is locally constant with fibre $\cH_0(\gamma_{t_o}^+\cN)$ (Proposition \ref{prop:locconstdegenerate}) and $U\cap S_o$ is simply connected, $(\cM^{t_o},\iso_{\wh0})$ extends in a unique way as a section of $\cH_T(\cN)_{|U\cap S_o}$, and we find $(\cM,\iso_{\wh T})$ defined in some neighbourhood of $U\cap S_o$ in~$T$. Given a stratum $S$ with $\ov S\supset\nobreak S_o$, $(\cM,\iso_{\wh T})$ is defined on some $V_S$ as above and defines a section of $\cH_T(\cN)_{|V_S\cap S}$. Since $\cH_T(\cN)_{|S}$ is locally constant and since \hbox{$\pi_1(V_S\cap S,\star)\to\nobreak\pi_1(U\cap\nobreak S,\star)$} is an isomorphism, this section extends to $U\cap S$ (\cf Corollary \ref{cor:extensiontop}) and we obtain $(\cM,\iso_{\wh T})$ on $U\cap S$, hence on a neighbourhood of $U\cap S$ in $U$. Since $U$ does not cut any stratum $S'$ such that $\ov{S'}\supsetneq S_o$, it is covered by the strata $U\cap S$ with $S$ such that $\ov S\supset S_o$. By uniqueness (Corollary \ref{cor:uniquenessglobalI}), the extensions on the neighbourhoods of the various strata glue together, and give rise to $(\cM,\iso_{\wh T})$ on $U$.
\end{proof}

\subsection{Proof of Corollaries \ref{cor:extension} and \ref{cor:extensionf}}
The first part of Corollary \ref{cor:extension} is contained in the theorem. For the second part, let $t_o\in V$ and let $(\cM,\iso_{\wh T})_V$ and $(\cM',\iso'_{\wh T})_V$ be two $\cN$-marked $T$-meromorphic flat bundles on $V\times(\CC_z,0)$ whose restrictions at $t_o$ are equal to $(\cM^{t_o},\iso_{\wh0})$. Then, by Theorem \ref{th:extension}, $(\cM,\iso_{\wh T})_V$ and $(\cM',\iso'_{\wh T})_V$ coincide on some $\nb(t_o)\subset T$, hence on a nonempty open set in $V\moins(\Delta\cap V)$. Since the latter is connected, they coincide on $V\moins(\Delta\cap V)$, by an argument similar to that used in Corollary \ref{cor:decomp}. We conclude with Lemma \ref{lem:Hartogs}.\qed

\begin{proof}[Proof of Corollary \ref{cor:extensionf}]
The surjectivity of~$\gamma_{y_o}^+$ and the injectivity of $f^+$ are obvious from Theorem \ref{th:extension}, since we can identify $(\cM^{y_o},\iso_{\wh0})$ with $(\cM^{t_o},\iso_{\wh0})$, due to the identification $\gamma_{y_o}^+\cN_Y=\gamma_{t_o}^+\cN$, and we have $\gamma_{t_o}^+=\gamma_{y_o}^+\circ f^+$. It is then enough to prove the injectivity of $\gamma_{y_o}^+$. This is an immediate consequence of Corollary \ref{cor:uniquenessglobalI}.
\end{proof}

\subsection{Application to the sheaf of Stokes torsors}\label{subsec:applStokes}
We note that the restriction of~$\cN$ to any stratum of the natural stratification of $T$ compatible with $\Delta$ is good, hence for each such stratum~$S$, the sheaf $\cH_S(i_S^+\cN)\simeq\St_S(i_S^+\cN)$ is a locally constant sheaf. On the other hand, we do not have much information on the sheaf $\St_T(\cN)$, except on the open dense stratum $T\moins\Delta$. However, the sheaf $\cH_T(\cN)$ is better behaved, and this will enable us to compare the various sheaves $\St_S(i_S^+\cN)$. Indeed, Theorem \ref{th:extension} can be interpreted as a constructibility theorem for the sheaf $\cH_T(\cN)$. 

Let us fix a stratum $S_o$ of $\Delta$ in~$T$, let $U$ be an open subset of $T$ containing $S_o$ and satisfying \ref{th:extension}\eqref{th:extensionb} and \eqref{th:extensionc}, and let us consider the following diagram:
\[
\xymatrix@C=1.5cm@R=.3cm{
&S_o\times\CC^m&\\
S_o=S_o\times\{0\}\ar@{^{ (}->}[r]^-{i_o}\ar[ddr]_{\id}& U\ar@{}[u]|\bigcup\ar[dd]^{p_o}&\ar@{_{ (}->}[l]_-{j_o}U\moins \Delta\ar[ddl]^{q_o}\\&&\\
&S_o&
}
\]
(\cf the notation in the proof of Proposition \ref{prop:locconstdegenerate}). We also denote by $i_o,j_o$ the complementary inclusions of $S_o$ and $T\moins S_o$ in $T$.

\begin{corollaire}[Constructibility of $\cH_T(\cN)$]\label{cor:constructible}
The sheaf $\cH_T(\cN)$ is constructible, and more precisely, for each stratum~$S_o$ of $\Delta$, each of the natural morphisms
\begin{starequation}\label{eq:poio}
p_{o*}\cH_T(\cN)_{|U}\to i_o^{-1}\cH_T(\cN)\To{i_o^+}\cH_{S_o}(i_{S_o}^+\cN)
\end{starequation}%
is an isomorphism (the left one comes from the sheaf-theoretic adjunction $\id\!\to\!i_{o*}i_o^{-1}$), and the natural composed morphism
\[
\cH_{S_o}(i_{S_o}^+\cN)\isom i_o^{-1}\cH_T(\cN)\to i_o^{-1}j_{o*}j_o^{-1}\cH_T(\cN)
\]
is injective. Moreover, the natural restriction morphism
\[
p_{o*}\cH_T(\cN)_{|U}\to q_{o*}\cH_T(\cN)_{|U\moins\Delta}
\]
is injective. Lastly, the natural morphism that one deduces from it together with \eqref{eq:poio} and the sheaf-theoretic adjunction $q_o^{-1}q_{o*}\to\id$:
\begin{starstarequation}\label{eq:poqo}
q_o^{-1}\cH_{S_o}(i_o^+\cN)\to\cH_T(\cN)_{|U\moins\Delta}
\end{starstarequation}%
is also injective. The image subsheaf of \eqref{eq:poqo} is characterized as follows: given $t\!\in\!U\moins\Delta$, setting \hbox{$t_o\!=\!q_o(t)\!\in\!S_o$}, a germ of section $(\cM^t,\iso_{\wh T,t})\in\cH_{T\moins\Delta}(\cN)_t\simeq\cH_0(i_t^+\cN)$ belongs to (the image of) $[q_o^{-1}\cH_{S_o}(i_o^+\cN)]_t\simeq\cH_{S_o}(i_o^+\cN)_{t_o}$ if and only if there exists a neighbourhood $\nb(t_o)$ of $t_o$ in $S_o$ and $(\cM,\iso_{\wh T})$ on $p_o^{-1}(\nb(t_o))$ whose germ at $t$ is~$(\cM^t,\iso_{\wh T,t})$.
\end{corollaire}

\begin{proof}
The statement is local on $S_o$. Each $t_o\in S_o$ has a fundamental system of $1$-connected neighbourhoods $V$ in $S_o$. On the other hand, each such $V$ has a fundamental system of neighbourhoods $U'$ in $T$ satisfying the properties \ref{th:extension}\eqref{th:extensiona}--\eqref{th:extensionc}. Let us denote by $i_o$ the inclusion $S_o\hto T$. We wish to prove that for each such pair $(V,U')$, the restriction map $i_o^+:\Gamma(U',\cH_T(\cN))\to\Gamma(V,\cH_{S_o}(i_o^+\cN))$ is a bijection, as the case $U'=p_o^{-1}(V)$ also implies that the composed morphism \eqref{eq:poio} is an isomorphism.

Since $V$ is simply connected and $\cH_{S_o}(i_o^+\cN)$ is locally constant (Proposition \ref{prop:locconstgeneric}), we have $\Gamma(V,\cH_{S_o}(i_o^+\cN))=i_{t_o}^{-1}\cH_{S_o}(i_o^+\cN)$. Moreover, the same proposition identifies $i_{t_o}^{-1}\cH_{S_o}(i_o^+\cN)$ with $\cH_0(i_{t_o}^+\cN)$ via $\gamma_{t_o}^+$. Lastly, Theorem \ref{th:extension} implies that $\gamma_{t^0}^+:\Gamma(U',\cH_T(N))\to\cH_0(i_{t_o}^+\cN)$ is a bijection. We conclude that $\Gamma(U',\cH_T(\cN))\to\Gamma(V,\cH_{S_o}(i_o^+\cN))$ is also a bijection. We then use Corollary \ref{cor:Hartogs} to conclude the first assertion.

For the second assertion, we need to prove that if $(\cM,\iso_{\wh T})$ and $(\cM',\iso'_{\wh T})$ defined on $p_o^{-1}(V)$ ($V$ as above) coincide on $q_o^{-1}(V)$, then they coincide. This follows from Hartogs theorem (Lemma \ref{lem:Hartogs}).

For the \eqref{eq:poqo}, we are left with proving that the adjunction morphism
\[
q_o^{-1}q_{o*}\cH_T(\cN)_{|U\moins\Delta}\to\cH_T(\cN)_{|U\moins\Delta}
\]
is injective. This amounts to proving that, if two sections of $\cH_T(\cN)$ on $q_o^{-1}(V)$ ($V,t_o$ as above) coincide at $t\in q_o^{-1}(t_o)$, they coincide everywhere on $q_o^{-1}(V)$. Since $q_o^{-1}(V)$ is also connected by the star-shaped property of $U$, we can apply the same argument as in Corollary \ref{cor:genericuniqueness} to obtain the desired property.

Let us prove the final assertion of the corollary. That $(\cM^t,\iso_{\wh T,t})\in\cH_{T\moins\Delta}(\cN)_t$ \hbox{belongs} to the image of $[q_o^{-1}\cH_{S_o}(i_o^+\cN)]_t$ means that there exists $(\cM^{t_o},\iso_{\wh T,t_o})\in\cH_{S_o}(i_o^+\cN)_{t_o}$, which we consider as a section $\sigma_o:S_o\supset\nb(t_o)\to\cH_{S_o}(i_o^+\cN)^\et$ of the sheaf (étalé) space $\cH_{S_o}(i_o^+\cN)^\et\to S_o$ (which is a covering since $\cH_{S_o}(i_o^+\cN)$ is locally constant) such that, if $\sigma:\nb(t)\to\cH_{T\moins\Delta}(\cN)^\et$ denotes the section corresponding to $(\cM^t,\iso_{\wh T,t})$, we have $\sigma(t)=(\sigma_o\circ q_o)(t)$, regarding $\sigma_o\circ q_o$ as a section
\[
q_o^{-1}(\nb(t_o))\to[q_o^{-1}\cH_{S_o}(i_o^+\cN)]^\et\to\cH_{T\moins\Delta}(\cN)^\et
\]
of the sheaf space $\cH_{T\moins\Delta}(\cN)^\et$ over $T\moins\Delta$ (which is also a covering). By~definition, $\sigma_o\circ q_o$ corresponds to the restriction to $q_o^{-1}(\nb(t_o))$ of $(\cM,\iso_{\wh T})$ constructed by Theorem \ref{th:extension} on $p_o^{-1}(\nb(t_o))$ from $(\cM^{t_o},\iso_{\wh0})$ corresponding to $\sigma_o$. Then the germ $(\cM,\iso_{\wh T})_t$ corresponds to $\sigma$ by the correspondence \eqref{eq:HSt}.

The converse is obtained similarly.
\end{proof}

On the one hand, let us consider the sheaf $\St_{S_o}(i_o^+\cN)=\St_{S_o}(i_o^+\cN_o)$ (\cf\eqref{eq:cNo}). Since $\cN_o$ is good on $S_o$, this is a locally constant sheaf of pointed sets on~$S_o$. On~the other hand, we have the locally constant sheaf $\St_{T\moins\Delta}(\cN)$ since $\cN$ is good on $T\moins\Delta$.

\begin{corollaire}[Comparison of sheaves of Stokes torsors]\label{cor:Stcomp}
There exist natural injective morphisms
\begin{starequation}\label{eq:localcomp}
\St_{S_o}(i_o^+\cN)\hto i_o^{-1}j_{o*}\St_{T\moins\Delta}(\cN).
\end{starequation}
and
\begin{starstarequation}\label{eq:qStoSt}
q_o^{-1}\St_{S_o}(i_o^+\cN)\hto \St_{T\moins\Delta}(\cN)_{|U\moins\Delta}.
\end{starstarequation}%
The image subsheaf of \eqref{eq:qStoSt} is characterized as follows: given $t\!\in\!U\moins\Delta$, setting \hbox{$t_o\!=\!q_o(t)\!\in\!S_o$}, a germ of section $\sigma_t\in\St_{T\moins\Delta}(\cN)_t\simeq\St_0(i_t^+\cN)$ belongs to (the image of) $[q_o^{-1}\St_{S_o}(i_o^+\cN)]_t\simeq\St_{S_o}(i_o^+\cN)_{t_o}$ if and only if there exists a neighbourhood $\nb(t_o)$ of $t_o$ in $S_o$ and $(\cM,\iso_{\wh T})$ on $p_o^{-1}(\nb(t_o))$ whose germ at $t$ corresponds to~$\sigma_t$ via \eqref{eq:HSt}.
\end{corollaire}

\begin{proof}
The morphism \eqref{eq:localcomp} is defined by the following diagram:
\[
\xymatrix@C=1.5cm{
\St_{S_o}(i_o^+\cN)\ar[d]_\wr\ar@{-->}[r]
&i_o^{-1}j_{o*}\St_{T\moins\Delta}(\cN)\ar[d]^\wr
\\
i_o^{-1}\cH_T(\cN)\ar@{^{ (}->}[r]
&i_o^{-1}j_{o*}j_o^{-1}\cH_T(\cN)
}
\]
The left vertical isomorphism is obtained by composing the isomorphism of Theorem~\ref{th:HSt} on $S_o$ ($\St_{S_o}(i_o^+\cN)\simeq\cH_{S_o}(i_o^+\cN)$) together with the isomorphism of Corollary~\ref{cor:constructible}. The right vertical isomorphism is given by Theorem \ref{th:HSt} on $T\moins\Delta$. Lastly, the lower horizontal morphism is injective, according to Corollary \ref{cor:Hartogs}. We define \eqref{eq:qStoSt} similarly and use \eqref{eq:poqo} for its injectivity. The last assertion is obtained similarly from the last assertion in Corollary~\ref{cor:constructible}.
\end{proof}

\section{Isomonodromic deformations}\label{sec:isomono}

\begin{notation}\mbox{}
\begin{enumerate}
\item
For a given square matrix $M$, we denote by $M=M^\diago+M^\nd$ its decomposition into the diagonal and non-diagonal parts.

\item
For $t_o\in\Delta$, we decompose $\{1,\dots,n\}=\bigsqcup_{r\in R}I_r$ such that, for every $r\in R$, we have $\{i,j\}\subset I_r$ if and only if $t_{o,i}=t_{o,j}$.
\end{enumerate}
\end{notation}

\subsection{Universal deformation of a Birkhoff normal form}
For $t\in T$, we denote by $\Lambda(t)$ the matrix $\diag(t_1,\dots,t_n)$. For $t_o\in T$, say that a system with matrix $(\sfrac{\Lambda(t_o)}{z}+B(z))\sfrac{\rd z}{z}$, with $B(z)$ holomorphic, is in \emph{Birkhoff normal form} if $B(z)=A_1$ is constant, that is,
\begin{equation}\label{eq:general}
\Bigl(\frac{\Lambda(t_o)}{z}+A_1\Bigr)\frac{\rd z}{z}.
\end{equation}

For the system \eqref{eq:general}, if $t_o\notin\Delta$, a theorem by B.\,Malgrange \cite{Malgrange83cb, Malgrange86} asserts that there exists a universal integrable defor\-mation of this system in the neighbourhood of $t_o$ (\cf also \cite[\S VI.3]{Bibi00b}). In particular (\cf \cite[\S VI.3.f]{Bibi00b}), there exists a holomorphic matrix $F^\nd_1(t)$ near $t_o$ with zeros on the diagonal, such that the system of the form~\eqref{eq:system}
\begin{equation}\label{eq:familyBirkhoff}
\Bigl(\frac{\Lambda(t)}{z}+[\Lambda(t),F^\nd_1(t)]+A^\diago_1\Bigr)\frac{\rd z}{z},
\end{equation}
where $A^\diago_1$ is the diagonal part of $A_1$, is integrable and restricts to \eqref{eq:general} at $t=t_o$. The integrable connection (on the trivial bundle) has matrix (\cf \cite[VI\,(3.12)]{Bibi00b})
\begin{equation}\label{eq:univintegr}
-\rd(\Lambda(t)/z)+\bigl([\Lambda(t),F_1^\nd(t)]+A_1^\diago\bigr)\,\frac{\rd z}{z}-[\rd\Lambda(t),F_1^\nd(t)]
\end{equation}
and is a universal integrable deformation of its restriction at each point of the neighbourhood where it exists. Moreover, there exists a $z$-formal base change that transforms \eqref{eq:univintegr} to the system
\begin{equation}\label{eq:univintegrform}
-\rd(\Lambda(t)/z)+A_1^\diago\,\frac{\rd z}{z}.
\end{equation}
The following two questions are natural:
\bgroup\let\theenumi\theequation
\par\smallskip
\refstepcounter{equation}\label{enum:univa}
\noindent\eqref{enum:univa}
If $t_o\in\Delta$, can we find an integrable deformation \eqref{eq:univintegr} of the Birkhoff normal form \eqref{eq:general} with $z$-formal normal form \eqref{eq:univintegrform}?
\par\smallskip
\refstepcounter{equation}\label{enum:univb}
\noindent\eqref{enum:univb}
Given a system \eqref{eq:system} defined on an open set $V$ of $T$ (so that its restriction at every $t\in V$ is in Birkhoff normal form), and given a matrix $F_1^\nd(t)$, holomorphic on $V$, assume that, in a small neighbourhood $W$ of $t_o\in V\moins\Delta$, \eqref{eq:system} underlies a~universal integrable deformation \eqref{eq:univintegr} of its restriction at $t=t_o$. Does this integrable deformation extend on $V$, in particular at points $t\in \Delta\cap V$, and has $z$-formal normal form \eqref{eq:univintegrform} on $V$?
\par\smallskip
\egroup

As an application of Theorem \ref{th:extension} we give an answer to \eqref{enum:univa} in Section \ref{subsec:defbirkhoff}. On the other hand, the results of \cite[Cor.\,1.1]{C-D-G17a} and \cite[Cor.\,2.1]{C-G17a} concern \eqref{enum:univb}.

\subsection{Deformation of a Birkhoff normal form}\label{subsec:defbirkhoff}

We will show that Question \eqref{enum:univa} for \eqref{eq:general} has a positive answer provided that Assumptions \ref{lem:decformto}\eqref{lem:decformtoa} and \eqref{lem:decformtob} below are fulfilled by $A_1$. Moreover, we make more precise the domain of existence of the integrable deformation.

Let $t_o\in \Delta\subset T$ and let us fix a Birkhoff normal form \eqref{eq:general}.

\begin{lemme}\label{lem:decformto}
With the following assumptions
\begin{enumeratea}
\item\label{lem:decformtoa}
$A_1^\nd\in\im\ad(\Lambda(t_o))$,\footnote{Recall that $\ad(\Lambda(t_o))(B)=[\Lambda(t_o),B]$ so that $A\in\im\ad(\Lambda(t_o))$ iff $A_{ij}=0$ whenever $i,j\in I_r$ for some $r\in R$.}
\item\label{lem:decformtob}
$A_1^\diago$ is partially non-resonant, \ie
\[
\forall r\in R,\;\forall i,j\in I_r,\quad A_1^\diago{}_{ii}-A_1^\diago{}_{jj}\notin\ZZ\moins\{0\},
\]
\end{enumeratea}
\noindent
which are always fulfilled if $t_o\notin\Delta$, there exists a $z$-formal base change of the form $\id+zP_1+\cdots$ that transforms the matrix of $\nabla^o$
\[
A(t_o,z)=\Big(\frac{\Lambda(t_o)}{z}+A_1\Big)\frac{\rd z}{z}
\]
to the diagonal matrix
\[
A^\diago(t_o,z)=\Big(\frac{\Lambda(t_o)}{z}+A_1^\diago\Big)\frac{\rd z}{z}.
\]
\end{lemme}

\begin{proof}
The result is standard (\cf \eg \cite[Prop.\,4.2]{C-D-G17a}), but we will give a proof valid on any algebraic closed field of characteristic zero instead of $\CC$, \ie not depending on the transcendental notion of fundamental solution. Firstly, one can find a formal base change $\wh F_o=\id-zF_1^\nd(t_o)+\cdots$ that transforms $A(t_o,z)$ to the formal matrix
\[
\wt A(t_o,z)=\Big(\frac{\Lambda(t_o)}{z}+A_1^\diago+zA_2+\cdots\Big)\frac{\rd z}{z},
\]
where $A_2,\dots$ are block diagonal with respect to $r\in R$, so that $\wt A(t_o,z)$ is also block diagonal. Each block ($r\in R$) is written as
\[
\wt A^{(r)}(t_o,z)=\Big(\frac{t_{o,r}\id^{(r)}}{z}+A_1^{\diago(r)}+zA^{(r)}_2+\cdots\Big)\frac{\rd z}{z}.
\]
Since $A_1^{\diago(r)}$ is non-resonant, there exists a base change $\id^{(r)}+zQ^{(r)}_1+\cdots$ that transforms $\wt A^{(r)}(t_o,z)$ to the $r$th block of $A^\diago(t_o,z)$.
\end{proof}

\begin{theoreme}\label{th:deformation}
Let $t_o\in\Delta$ and let $U$ be a connected open subset of $T$ satisfying \hbox{\ref{th:extension}\eqref{th:extensiona}--\eqref{th:extensionc}} with respect to $S(t_o)$. Under Assumptions \ref{lem:decformto}\eqref{lem:decformtoa} and \eqref{lem:decformtob} on $A_1$, there exists a holomorphic hypersurface~$\Theta$ in $U$ not containing $t_o$ and a holomorphic matrix $F_1^\nd(t)$ on $U\moins\Theta$, such that the meromorphic connection on the trivial bundle with matrix \eqref{eq:univintegr} is integrable, restricts to \eqref{eq:general} at $t_o$, and is formally equivalent at $z=0$ to the matrix connection
\[
-\rd(\Lambda(t)/z)+A_1^\diago\,\frac{\rd z}{z}.
\]
\end{theoreme}

\begin{proof}
The proof is very similar to that of \cite{Malgrange83cb, Malgrange86}. We set $T=\{z=0\}\subset X=\CC^n\times\CC_z$. We \hbox{denote} by~$\cN$ the $T$\nobreakdash-meromorphic flat bundle $\bigoplus_{i=1}^n(\cE^{-t_i/z}\otimes\cR_i)$, where $\cR_i\!=\!(\cO_X(*T),\rd+\nobreak A_1^\diago{}_{ii}\,\rd z/z)$. By Lemma \ref{lem:decformto}, \eqref{eq:general} defines an object $(\cM^{t_o},\iso_{\wh0})$ in $\cH_0(\gamma_{t_o}^+\cN)$. From Theorem \ref{th:extension} we deduce an $\cN$-marked $T$-meromorphic flat bundle $(\cM,\iso_{\wh T})$ on $U\times(\CC_z,0)$. We can now apply \cite[Th.\,VI.2.1]{Bibi00b} and obtain a hypersurface $\Theta\subset U$ and a basis~$\epsilong$ of $\cM(*(\Theta\times\CC_z))$ in which the matrix of the integrable connection takes the form
\[
\Big(\frac{A_0(t)}{z}+A_1\Big)\frac{\rd z}{z}+\frac{C(t)}{z}
\]
with $A_0(t)$ conjugate to $\Lambda(t)$ for each $t\in U\moins\Theta$ and $C(t)\in\Gamma(U,\Omega^1_T(*\Theta))$. One can then apply the results in \cite[\S VI.3.f]{Bibi00b}, since the regularity property of~$\Lambda(t)$ (\ie the fact that the eigenvalues are pairwise distinct) is not needed at this point. The change of notation with respect to \loccit\ is as follows: $\Delta_0$ is $\Lambda(t)$, $\Delta_\infty$ is $A_1^\diago$, $\tau$ is $z$ and the non-diagonal part $T^\nd$ is $F_1^\nd(t)$.
\end{proof}

\subsection{Application to the construction of Frobenius manifolds}
It is known (\cf \cite[Th.\,3.2,\,p.\,223]{Dubrovin96}) that, given suitable initial data consisting of a diagonal matrix $\Lambda(t_o)$ with pairwise distinct eigenvalues (\ie $t_o\notin\Delta$), of a matrix~$A_1$ such that $A_1-(w/2)\id$ is skew-symmetric for some integer $w$, and an eigenvector~$\omega_o$ of~$A_1$ that has no zero entry, one can construct a Frobenius manifold structure on the complement of some hypersurface $\Theta$ in the universal covering $\wt{T\moins\Delta}$ (\cf also \cite[\S VII.4.a]{Bibi00b}). It is the universal model at a semisimple point of a Frobenius manifold.

Theorem \ref{th:deformation} enables us to relax the regularity assumption on $\Lambda(t_o)$, provided that Assumptions \ref{lem:decformto}\eqref{lem:decformtoa} and \eqref{lem:decformtob} are fulfilled for $A_1$. Let $t_o\in\Delta$ and set $S_o=S(t_o)$. Assume that $(\Lambda(t_o),A_1,\omega_o)$ satisfy the following properties.
\begin{enumeratea}
\item\label{enum:a}
$A_1^\nd\in\im\ad(\Lambda(t_o))$,
\item\label{enum:b}
$A_1-(w/2)\id$ is skew-symmetric (so that $A_1^\diago=(w/2)\id$ is non-resonant),
\item\label{enum:c}
$\omega_o$ is an eigenvector of $A_1$ whose entries are all nonzero.
\end{enumeratea}
Let $U$ be an open subset of $T$ containing $S_o$ and satisfying \ref{th:extension}\eqref{th:extensionb} and \eqref{th:extensionc}, and let~$\wt U$ be its universal covering. By Theorem \ref{th:deformation}, there exists an integrable deformation of the Birkhoff normal form \eqref{eq:general}, which exists on $\wt U\moins\Theta$ for some complex hypersurface~$\Theta$ of $\wt U$. The vector $\omega_o$ can be extended flatly as a vector function $\omega$ on~$\wt U$, meromorphic along $\Theta$ and its entries do not vanish away from some hypersurface~$\Theta_{\omega_o}$. To $\omega$ is then associated an infinitesimal period mapping (\cf\cite[\S VII\,3.a]{Bibi00b}).

\begin{corollaire}\label{cor:410}
Under Assumptions \eqref{enum:a}, \eqref{enum:b} and \eqref{enum:c} above, the infinitesimal period mapping associated with $\omega$ endows the manifold $\wt U\moins(\Theta\cup\Theta_{\omega_o})$ with a Frobenius structure.\qed
\end{corollaire}

\backmatter
\providecommand{\SortNoop}[1]{}\providecommand{\sortnoop}[1]{}\providecommand{\eprint}[1]{\href{http://arxiv.org/abs/#1}{\texttt{arXiv\string:\allowbreak#1}}}\providecommand{\hal}[1]{\href{https://hal.archives-ouvertes.fr/hal-#1}{\texttt{hal-#1}}}\providecommand{\tel}[1]{\href{https://hal.archives-ouvertes.fr/tel-#1}{\texttt{tel-#1}}}\providecommand{\doi}[1]{\href{http://dx.doi.org/#1}{\texttt{doi\string:\allowbreak#1}}}\providecommand{\eprint}[1]{\href{http://arxiv.org/abs/#1}{\texttt{arXiv\string:\allowbreak#1}}}\providecommand{\doi}[1]{\href{http://dx.doi.org/#1}{\texttt{doi\string:\allowbreak#1}}}
\providecommand{\bysame}{\leavevmode ---\ }
\providecommand{\og}{``}
\providecommand{\fg}{''}
\providecommand{\smfandname}{\&}
\providecommand{\smfedsname}{\'eds.}
\providecommand{\smfedname}{\'ed.}
\providecommand{\smfmastersthesisname}{M\'emoire}
\providecommand{\smfphdthesisname}{Th\`ese}

\bgroup\smaller\renewcommand{\baselinestretch}{.9}\normalfont
\subsubsection*{Added on April 21, 2021}
Concerning question \eqref{enum:univb}, a new proof has been given by D.\,Guzzetti (Isomonodromic Laplace transform with coalescing eigenvalues and confluence of Fuchsian singularities, \eprint{2101.03397}), and another proof has later been given by the author (A short proof of a theorem of Cotti, Dubrovin and Guzzetti, \eprint{2103.16878}). Concerning Corollary \ref{cor:410}, another proof has been given by G.\,Cotti (Degenerate Riemann-Hilbert-Birkhoff problems, semisimplicity, and convergence of WDVV-potentials, \eprint{2011.04498}).
\egroup
\end{document}